\documentclass[11pt]{amsart}
\usepackage[latin1]{inputenc}
\usepackage{amsmath,amsthm,amsfonts,amssymb,amscd}
\usepackage{mathrsfs}
\usepackage{verbatim}
\usepackage{multicol}
\usepackage{tabularx}
\usepackage[usenames]{color}
\usepackage{enumerate}
\usepackage{graphicx}
\usepackage{bbm}
\usepackage[all]{xy}
\usepackage{mathrsfs} 
\usepackage{tikz}
\usetikzlibrary{arrows,matrix,decorations.markings,shapes.geometric}

\newtheorem{thm}{Theorem}[section]
\newtheorem{prop}[thm]{Proposition}
\newtheorem{cor}[thm]{Corollary}

\newtheorem{lemma}[thm]{Lemma}

\theoremstyle{definition}
\newtheorem{rem}[thm]{Remark}

\numberwithin{equation}{section}

\newcommand{\m}{\mathsf m}
\def\cal#1{\text{$\mathcal{#1}$}}
\def\scr#1{\text{$\mathscr{#1}$}}
\def\ord#1^#2{#1$^{\text{#2}}$}
\def\lie#1{\mathfrak{#1}}
\def\ol#1{\overline{#1}}

\def\hlie#1{\hat{\mathfrak{#1}}}
\def\uq#1{\text{$U_q(\lie #1)$}}
\def\uqr#1^#2{\text{$U_q^{#2}(\lie #1)$}}
\def\uqh#1{\text{$U_q(\hlie #1)$}}
\def\uqhr#1^#2{\text{$U_q^{#2}(\hlie #1)$}}
\def\us#1^#2{\text{$U_{\xi}^{#2}(\lie #1)$}}
\def\ush#1^#2{\text{$U_{\xi}^{#2}(\hlie #1)$}}
\def\dus#1^#2{\text{$\dot{U}_{\xi}^{#2}(\lie #1)$}}
\def\dush#1^#2{\text{$\dot{U}_{\xi}^{#2}(\hlie #1)$}}

\def\gbr#1{{\mbox{\boldmath ${\rm #1}$}}}

\def\wt{{\rm wt}}

\def\Rep{{\rm Rep}}
\def\res{{\rm res}}
\def\opl_#1^#2{\text{\scriptsize$\bigoplus\limits_{\text{\normalsize$#1$}}^{\text{\normalsize$#2$}}$}}
\def\otm_#1^#2{\text{\scriptsize$\bigotimes\limits_{\text{\footnotesize$#1$}}^{\text{\footnotesize$#2$}}$}}

\def\tqbinom#1#2{\text{$\left[\begin{smallmatrix} #1\\#2\end{smallmatrix}\right]$}}
\def\bs#1{\boldsymbol{#1}}
\newcommand{\nops}{\overline{\mathscr P}_{(i_t,k_t)_{1\leq t\leq T}}}
\newcommand{\ps}{{(p_1,\dots,p_T)}}
\newcommand{\bt}{\mathsf b}
\def\gbr#1{\boldsymbol{{\rm #1}}}
\def\endit{\hfill$\diamond$}

\newcommand{\g}{\mathfrak{g}}
\newcommand{\h}{\mathfrak{h}}

\newcommand{\C}{\mathbb{C}}

\newcommand{\Z}{\mathbb{Z}}
\newcommand{\R}{\mathbb{R}}
\newcommand{\N}{\mathbb{N}}

\newcommand{\blue}{\color{\blue}}

\def\scr#1{\mathscr{#1}}

\renewcommand{\thefootnote}

\newcommand{\sd}{(\lambda/\mu)}
\newcommand{\Alp}{\mathsf A}
\DeclareMathOperator{\Tab}{Tab}
\renewcommand{\L}{L}

\begin{document}

\title[Representations of quantum affine algebras of type $B_N$]{Representations of quantum affine algebras of type $B_N$}
\author{Matheus Brito and Evgeny Mukhin}

\address{MB: Departamento de Matemática, Universidade Estadual de Campinas, Campinas - SP - Brazil, 13083-859.}
\email{mbrito@ime.unicamp.br}

\address{EM:  402 N. Blackford St, LD 270, IUPUI, Indianapolis, IN 46202, USA.}
\email{mukhin@math.iupui.edu}

\begin{abstract} We study finite-dimensional representations of quantum affine algebras of type $B_N$. We show that a module is tame if and only if it is thin. In other words, the Cartan currents are diagonalizable if and only if all joint generalized eigenspaces 
have dimension one. We classify all such modules and describe their $q$-characters. In some cases, the $q$-characters are described by super standard Young tableaux of type $(2N|1)$.
\end{abstract}

\maketitle

\setcounter{section}{0}

\section{Introduction} It is well-known that the study of finite-dimensional modules over affine quantum groups is a rich and difficult subject. Many results are available but even the structure of irreducible representations is largely out of reach. However, there are families of representations which are better understood. For example, in type $A$, one has evaluation representations. Their analogs in other types, called minimal affinizations, received a lot of attention, see \cite{C95, CP95, CP96a, CP96b, H07, LM13, MTZ04, M10, N13, N14}. The study becomes much easier when the module is "thin", meaning that the Cartan elements act diagonally with simple joint spectrum. For example, this occurs for all minimal affinizations in types $A$ and $B$. In such cases, one often can describe the combinatorics of the module explicitly.

The motivation for this paper comes from the works \cite{NT98}, \cite{MY12a} and \cite{KOS95}. It is shown in \cite{NT98} in type $A$, that if the Cartan generators are diagonalizable in an irreducible module (we call this property "tame"), then their joint spectrum is necessarily simple (that is thin). Moreover, all such modules are pull-backs with respect to the evaluation homomorphism from a natural class of $U_q(\hlie{sl}_N)$-modules and their $q$-characters are described by the semistandard Young tableaux corresponding to fixed skew Young diagrams.
In this paper we extend these results to type $B$. We are assisted by \cite{MY12a}, where the $q$-characters of a large family of thin $B_N$ modules are described combinatorially in terms of certain paths and by \cite{KOS95}, where some of the $q$-characters are given in terms of certain Young tableaux.

We define explicitly a family of sets of Drinfeld polynomials which we call "extended snakes", (see Section \ref{s:extsnake}) and consider the corresponding irreducible finite-dimensional modules of quantum affine algebra of type $B_N$. This family contains all snake modules of \cite{MY12a}, in particular, it contains all minimal affinizations. We extend the methods of \cite{MY12a} and describe the $q$-characters of the extended snake modules via explicit combinatorics of paths, see Theorem \ref{mainthm1}. It is done by the use of the recursive algorithm of \cite{FM01}, since the extended snake modules are thin and special (meaning that there is only one dominant affine weight).

Then we show that a simple tame module of $B_N$ type has to be an extended snake module (more precisely, a tensor product of snake modules, see Section \ref{chi sec}), see Theorem \ref{mainthm2}. It is done by the reduction to the results of \cite{NT98} and by the induction on $N$. It turns out that it is sufficient to control only a small part of the $q$-character near the highest weight monomial.

Therefore, we obtain the main result of the paper: an irreducible module in type $B$ is tame if and only if it is thin. 
All such modules are special and antispecial. Moreover, thin modules are (tensor products of) extended snake modules and their $q$-characters are described explicitly.

Finally, we study the combinatorics of tame $B_N$ modules in terms of Young tableaux. We observe a curious coincidence with the representation theory of superalgebra $\mathfrak{gl}(2N|1)$. The irreducible representations of the latter algebra are parametrized by Young diagrams which do not contain the box with coordinates $(2N+1,2)$. More generically, one can construct representations of $\mathfrak{gl}(2N|1)$ corresponding to skew Young diagrams which do not contain a rectangle with vertical side of length $2N+1$ and horizontal side of length $2$. The character of such representations is given by super semistandard Young tableaux, see \cite{BR83}. We find that each such skew Young diagram there also corresponds to an irreducible snake module of the affine quantum algebra of type $B_N$. Moreover, the $q$-character of this module is described by the same super standard Young tableaux, see Theorem \ref{t:tabxpaths}. Note that not all snake modules appear that way and there are cases when two different skew Young diagrams correspond to the same snake module, see Section \ref{non generic sec}. We have no conceptual explanation for this coincidence.

\medskip

We expect that a similar analysis by the same methods can be done in other types 
and that the properties of being thin and tame are equivalent in general. 
In particular, one has Young tableaux description of certain modules in types $C$ and $D$, see \cite{NN06, NN07a, NN07b}. Note, however, in other types, minimal affinizations are neither thin nor special in general, see \cite{H07}, \cite{LM13}. 

\medskip

The paper is constructed as follows. In Section \ref{s:basics}  we introduce quantum affine algebras and the relevant notation. In Section \ref{s:thin tame} we describe the problem, recall the answer from \cite{NT98} in type $A$ and recover some of the proofs to illustrate our methods.
In Section \ref{s:extsnake}, we define the extended snake modules and follow the techniques of \cite{MY12a} to compute their $q$-characters. In Section \ref{s:tame}, we consider tame modules and show that they are the extended snake modules. In Section \ref{s:tableaux} we study the bijection between paths and super standard skew Young tableaux.

\medskip

{\bf {Acknowledgements.}}  MB is grateful to the Department of Mathematical Sciences, IUPUI, for their hospitality during two visits when the majority of this research was carried out. MB was supported by FAPESP, grants 2010/19458-9 and 2012/04656-5.

\section{Background}\label{s:basics}
We recall basic facts about finite-dimensional representations of quantum affine algebras and set up our notation.
\subsection{Cartan data.} Let $\g$ be a complex simple Lie algebra of rank $N$ and $\h$ a Cartan subalgebra of $\g$. We identify $\h$ and $\h^*$ by means of the invariant inner product $\left\langle\cdot, \cdot \right\rangle$ on $\g$ normalized such that the square length of the maximal root equals $2$. Let $I = \{1, \ldots, N\}$ and let $\{\alpha_i\}_{i\in I}$ be a set of simple roots, $\{\alpha_i^{\vee}\}_{i\in I}$ and $\{\omega_i\}_{i\in I}$, the sets of, respectively, simple coroots and fundamental weights. Let $C = (c_{ij})_{i,j\in I}$ denote the Cartan matrix. We have
$$2\left\langle \alpha_i, \alpha_j\right\rangle = c_{ij} \left\langle \alpha_i, \alpha_i \right\rangle,\qquad 2\left\langle \alpha_i, \omega_j\right\rangle = \delta_{ij} \left\langle \alpha_i, \alpha_i \right\rangle.$$
Let $r^{\vee}$ be the maximal number of edges connecting two vertices of the Dynkin diagram of $\g$. Let $r_i = \frac{1}{2}r^{\vee}\left\langle \alpha_i, \alpha_i \right\rangle$. We set 
$$D := {\rm diag}(r_1, \ldots, r_N),$$
so that $DC$ is symmetric.

Let $Q$ (resp. $Q^+$) and $P$ (resp. $P^+$) denote the $\Z$-span (resp. $\Z_{\geq 0}$-span) of the simple roots and fundamental weights respectively. Let $\leq$ be the partial order on $P$ in which $\lambda \leq \lambda'$ if and only if $\lambda'-\lambda \in Q^+$.

Let $\hlie g$ be the untwisted affine algebra corresponding to $\g$.

For $p\in \C$, $r,m\in\mathbb Z_{\ge 0}$, $m\ge r$, define
\begin{equation*}
[m]_p=\frac{p^m -p^{-m}}{p -p^{-1}},\ \ \ \ [m]_p! =[m]_p[m-1]_p\ldots
[2]_p[1]_p,\ \ \ \ \tqbinom{m}{r}_p = \frac{[m]_p!}{[r]_p![m-r]_p!}.
\end{equation*}

\subsection{Quantum Affine Algebras} Fix $q\in \C^*$, not a root of unity. Let $q_i := q^{r_i}$ for $i\in I$. The \emph{quantum affine algebra} $\uqh g$ in Drinfeld's new realization, \cite{Dri88}, is generated by $x_{i,r}^{\pm}$, $k_i^{\pm}$, $h_{i,s}$, for $i\in I$, $r\in \Z$ and $s\in \Z\setminus\{0\}$, and central elements $c^{\pm 1/2}$, subject to the following relations:
\begin{align*}
k_ik_i^{-1} = k_i^{-1}k_i =1, \ \  k_ik_j &=k_jk_i,\ \ k_ih_{j,r}= h_{j,r}k_i,\\
k_ix_{j,r}^\pm k_i^{-1} = q_i^{{}\pm c_{ij}}x_{j,r}^{{}\pm{}},\ \ &[h_{i,r} , x_{j,s}^{{}\pm{}}] = \pm\frac1r[rc_{ij}]_{q_i}x_{j,r+s}^{{}\pm{}},\\
x_{i,r+1}^{{}\pm{}}x_{j,s}^{{}\pm{}} -q_i^{{}\pm
c_{ij}}x_{j,s}^{{}\pm{}}x_{i,r+1}^{{}\pm{}} &
=q_i^{{}\pm     c_{ij}}x_{i,r}^{{}\pm{}}x_{j,s+1}^{{}\pm{}} -x_{j,s+1}^{{}\pm{}}x_{i,r}^{{}\pm{}},\\
[h_{i,n}, h_{j,m}] = \delta_{n,-m} \frac{1}{n}[nc_{ij}]_{q_i}\dfrac{c^n - c^{-n}}{q - q^{-1}}, \ \ & 
[x_{i,r}^+ , x_{j,s}^-]= \delta_{i,j} \frac{c^{(r-s)/2}\phi_{i,r+s}^+ - c^{-(r-s)/2}\phi_{i,r+s}^-}{q_i - q_i^{-1}},\\
\sum_{\sigma\in S_m}\sum_{k=0}^m(-1)^k \tqbinom{m}{k}_{q_i}x_{i,
n_{\sigma(1)}}^{{}\pm{}}\ldots x_{i,n_{\sigma(k)}}^{{}\pm{}} &
x_{j,s}^{{}\pm{}} x_{i, n_{\sigma(k+1)}}^{{}\pm{}}\ldots
x_{i,n_{\sigma(m)}}^{{}\pm{}} =0,\ \ \text{if $i\ne j$},
\end{align*}
for all sequences of integers $n_1,\ldots, n_m$, where $m
=1-c_{ij}$, $S_m$ is the symmetric group on $m$ letters, and the
$\phi_{i,r}^{{}\pm{}}$ are determined by 
\begin{equation}
\phi_i^\pm(u) = \sum_{r=0}^{\infty}\phi_{i,\pm r}^{\pm}u^{\pm r} =
k_i^{\pm 1} \exp\left(\pm(q-q^{-1})\sum_{s=1}^{\infty}h_{i,\pm s}
u^{\pm s}\right).
\label{phi def}
\end{equation}

Consider the subalgebras $\uqh {n^{\pm}}$ and $\uqh h$ of $\uqh g$ generated, respectively, by $(x_{i,r}^{\pm})_{i\in I, r\in \Z}$ and $(h_{i,s})_{i\in I, s\in \Z\setminus\{0\}}, (k_i)_{i\in I}$ and $c^{\pm 1/2}$. We have the isomorphism of vector spaces
$$\uqh g \cong \uqh {n^-}\otimes \uqh h \otimes \uqh {n^+}.$$

There exist a coproduct, a counit, and an antipode making $\uqh g$ a Hopf algebra. An explicit formula for
the comultiplication of the current generators of $U_q(\hlie g)$ is not known. However, we have the following useful lemma.

\begin{lemma}[\cite{Da98}]\label{e:coprod}
Modulo $U_q(\hlie g)X^-\otimes U_q(\hlie g)X^+$ we have
$$\Delta(\phi_i^{\pm}(u)) = \phi_i^{\pm}(u)\otimes \phi_i^{\pm}(u), $$
where $X^{\pm}$ is the $\C$-span of the elements $x_{j,r}^{\pm}$, $j\in I, r\in \Z$. \qed
\end{lemma}

The subalgebra of $\uqh g$ generated by $(k_i)_{i\in I}$, $(x_{i,0}^{\pm})_{i\in I}$ is a Hopf sublagebra of $\uqh g$ and is isomorphic as a Hopf algebra to $\uq g$, the quantized enveloping algebra of $\g$.

\subsection{Finite-dimensional representations and $q$-characters}
 A representation $V$ of $\uqh g$ is said to be of type $1$ if $c^{\pm 1/2}$ acts as the identity on $V$ and 
$$V = \bigoplus_{\lambda\in P} V_{\lambda}, \qquad V_{\lambda} = \{v\in V | k_iv = q^{\left\langle \alpha_i , \lambda \right\rangle}v\}.$$ 
Throughout this paper all representations will be assumed to be finite-dimensional representations of type $1$. 

The above decomposition of a module $V$ into its $\uq g$-weight spaces can be refined by decomposing it into Jordan subspaces of mutually commuting $\phi_{i,\pm r}^{\pm}$ defined in (\ref{phi def}), \cite{Kn95}, \cite{FR98}:
\begin{equation}
V = \bigoplus_{\bs\gamma}V_{\bs\gamma}, \qquad \bs\gamma = (\gamma_{i,\pm r}^{\pm})_{i\in I,\ r\in\Z_{\geq 0}}, \quad \gamma_{i,\pm r}^{\pm} \in \C,
\label{lw space decomposition}
\end{equation}
where
$$V_{\bs\gamma} = \left\{v\in V |\, \exists k\in \N, \forall i\in I,\ m\geq 0 \ \left(\phi_{i,\pm m}^{\pm} - \gamma_{i,\pm m}^{\pm}\right)^k v = 0\right\}.$$
If $\dim(V_{\bs\gamma})> 0$, $\bs\gamma$ is called an \emph{$\ell$-weight} of $V$. For every finite-dimensional representation of $\uqh g$, the $\ell$-weights are known \cite{FR98} to be of the form
$$\gamma_i^{\pm}: = \sum_{r=0}^{\infty}\gamma_{i,\pm r}^{\pm}u^{\pm r} = q_i^{\deg Q_i - \deg R_i}\dfrac{Q_i(uq_i^{-1})R_i(uq_i)}{Q_i(uq_i)R_i(uq_i^{-1})},$$
where the right hand side for $\gamma_i^+$ (resp. $\gamma_i^-$) is treated as a formal series in positive (resp. negative) integer powers of $u$, and $Q_i$ and $R_i$ are polynomials of the form
$$Q_i(u) = \prod_{a\in \C^*}(1-ua)^{w_{i,a}}, \qquad R_i(u) = \prod_{a\in \C^*}(1-ua)^{x_{i,a}},$$
for some $w_{i,a}, x_{i,a}\geq 0$, $i\in I$. Let $\cal P$ denote the free abelian multiplicative group of monomials in infinitely many formal variables $(Y_{i,a})_{i\in I, \ a\in \C^*}$. The group $\cal P$ is in bijection with the set of $\ell$-weights $\bs\gamma$ of the form above according to
\begin{equation}
\bs\gamma(m) \ \ {\rm with}\ \ m=\prod_{i\in I, a\in \C^*}Y_{i,a}^{w_{i,a}-x_{i,a}}.
\end{equation}
We identify elements of $\cal P$ with $\ell$-weights of finite-dimensional representations in this way, and henceforth write $V_m$ for $V_{\bs\gamma(m)}$. Let $\Z\cal P = \Z\left[Y_{i,a}^{\pm 1}\right]_{i\in I, a\in \C^*}$ be the ring of Laurent polynomials in the variables $Y_{i,a}$ with integer coefficients. The $q$-character map $\chi_q$, \cite{FR98}, is the injective homomorphism of rings
$$\chi_q: \Rep \uqh g \rightarrow \Z\left[ Y_{i,a}^{\pm 1}\right]_{i\in I, a\in \C^*},$$
defined by
$$\chi_q(V) = \sum_{m\in \cal P}\dim (V_m)m,$$
where $\Rep \uqh g$ is the Grothendieck ring of finite-dimensional representations of $\uqh g$. 

For any finite-dimensional representation $V$ of $\uqh g$, define
$$\cal M(V):= \{m\in \cal P | \dim(V_m)>0\}\subset \cal P.$$

For each $j\in I$, a monomial $m = \prod_{i\in I, a\in \C^*}Y_{i,a}^{u_{i,a}}$ is said to be \emph{$j$-dominant} (resp. \emph{$j$-anti-dominant}) if $u_{j,a}\geq 0$ (resp. $u_{j,a}\leq 0$) for all $a\in \C^*$. A monomial is \emph{(anti-)dominant} if it is $i$-(anti-)dominant for all $i\in I$. Let $\cal P^+\subseteq \cal P$ denote the submonoid of all dominant monomials.

Given a $\uqh g$-module $V$, a vector $v\in V$ is called a \emph{highest $\ell$-weight vector}, with \emph{highest $\ell$-weight} $\bs\gamma$, if $$\phi_{i,\pm s}^{\pm} v = \gamma_{i,\pm s}^{\pm} v \quad {\rm and}\quad x_{i,r}^{+}v = 0, \quad {\rm for\ all}\ i\in I, r\in \Z, s\in \Z_{\geq 0}.$$
A representation $V$ of $\uqh g$ is said to be a \emph{highest $\ell$-weight representation} if $V = \uqh g v$ for some highest $\ell$-weight vector $v\in V$.
Similarly, we define \emph{lowest $\ell$-weight vectors} and \emph{lowest $\ell$-weight representations}. 

For every $m\in \cal P^+$ there is a unique finite-dimensional irreducible representation of $\uqh g$, that is highest $\ell$-weight with highest $\ell$-weight $\bs\gamma(m)$, and moreover every finite-dimensional irreducible $\uqh g$-module is of this form for some $m\in \cal P^+$ \cite{CP94b}. We denote the irreducible module corresponding to $m\in\cal P^+$ by $L(m)$.
 
A finite-dimensional representation of $\uqh g$ $V$ is said to be \emph{special} (resp. \emph{anti-special}) if $\chi_q(V)$ has exactly one dominant (resp. anti-dominant) monomial. It is \emph{tame} if the action of $\uqh h$ on $V$ is semisimple. It is \emph{thin} if $\dim(V_m)\leq 1$ for all monomials $m$. We observe that if $V$ is thin then it is also tame.   

We let $\wt:\cal P \rightarrow P$ be the group homomorphism defined by $\wt(Y_{i,a}) = \omega_i$, $i\in I, a\in \C^*$. Define $A_{i,a}\in \cal P$, $i\in I, a \in \C^*$, by
\begin{equation}\label{e:lroots}
A_{i,a} = Y_{i,aq_i}Y_{i,aq_i^{-1}}\prod_{c_{ji}=-1}Y_{j,a}^{-1}\prod_{c_{ji}=-2}Y_{j,aq}^{-1}Y_{j,aq^{-1}}^{-1}\prod_{c_{ji}=-3}Y_{j,aq^2}^{-1}Y_{j,a}^{-1}Y_{j,aq^{-2}}^{-1}.
\end{equation}
Let $\cal Q$ be the subgroup of $\cal P$ generated by $A_{i,a}$, $i\in I, a \in \C^*$. Let $\cal Q^{\pm}$ be the monoid generated by $A_{i,a}^{\pm 1}$, $i\in I, a \in \C^*$. Note that $\wt(A_{i,a}) = \alpha_i$. There is a partial order $\leq$ on $\cal P$ in which $m\leq m'$ implies $m'm^{-1} \in \cal Q^+$. Moreover, this partial order is compatible with the partial order on $P$ in the sense that $m\leq m'$ implies $\wt (m) \leq \wt (m')$. For $j\in I$, let $\cal Q_j^{\pm}$ be the submonoid of $\cal Q$ generated by $A_{j,a}^{\pm}$, $a\in \C^*$.

For all $m\in \cal P^+$,
\begin{equation}
\cal M(L(m)) \subseteq m\cal Q^-,
\label{mon set}
\end{equation}
see \cite{FR98}, \cite{FM01}.

For all $i\in I, a\in \C^*$, let $u_{i,a}$ be the homomorphism of abelian groups $\cal P \rightarrow \Z$ such that
$$u_{i,a}(Y_{j,b})=\left\{\begin{array}{ll}
1& i=j \ {\rm and}\ a=b,\\
0 & {\rm otherwise}.
\end{array}\right.$$

For each $J\subseteq I$ we denote by $U_{q}(\hlie {g}_J)$ the subalgebra of $\uqh g$ generated by $c^{\pm 1/2}$, $(x_{j,r}^{\pm})_{j\in J, r\in \Z}$, $(\phi_{j,\pm r}^{\pm})_{j\in J, r\in \Z_{\geq 0}}$. Let $\cal P_J$ be the subgroup of $\cal P$ generated by $(Y_{j,a})_{j\in J, a\in \C^*}$ and $\cal P^+_J\subseteq \cal P_J$ the set of $J$-dominant monomials. We use the notation $\cal P_j$ and $\cal P_j^+$ when $J=\{j\}$ is the set of a single vertex $j\in I$. Let 
$${\rm res_J}: \Rep \uqh g \rightarrow \Rep U_q(\hlie g_J),$$
be the restriction map and $\beta_J$ be the homomorphism $\cal P \rightarrow \cal P_J$ sending $Y_{i,a}^{\pm}$ to itself for $i\in J$ and to $1$ for $i\notin J$. It is well known \cite{FR98} that the following diagram is commutative
$$\xymatrix{
\Rep \uqh g \ar[r]^-{\chi_q}\ar[d]^{\res_J} & \Z[Y_{i,a}^{\pm 1}]_{i\in I, a\in\C^*}\ar[d]^-{\beta_J}\\
\Rep U_q(\hlie g_J)\ar[r]^-{\chi_q} & \Z[Y_{j,b}^{\pm 1}]_{j\in J, b\in\C^*}.}
$$
As shown in \cite{FM01}, for each $J\subseteq I$ there exists a ring homomorphism 
$$\tau_J : \Z[Y_{i,a}^{\pm 1}]_{i\in I, a\in \C^*} \rightarrow \Z[Y_{j,b}^{\pm 1}]_{j\in J, b\in \C^*}\otimes \Z[Z_{k,c}^{\pm}]_{k\in I \setminus J, c \in \C^*},$$
where $(Z_{k,c}^{\pm})_{k\in I\setminus J, c\in \C^*}$ are certain new formal variables, with the following properties:
\begin{enumerate}
\item $\tau_J$ is injective.
\item $\tau_J$ refines $\beta_J$ in the sense that $\beta_J$ is the composition of $\tau_J$ with the homomorphism $\cal P_J\otimes \Z[Z_{k,c}^{\pm 1}]_{k\in I\setminus J, c\in \C^*}\rightarrow \cal P_J$ which sends $Z_{k,c}\mapsto 1$ for all $k\notin J$, $c\in \C^*$. Moreover, the restriction of $\tau_J$ to the image of $\Rep \uqh g$ in $\Z[Y_{i,a}^{\pm 1}]_{i\in I, a \in \C^*}$ is a refinement of the restriction homomorphism $\res_J$.
\item When $J=\{j\}\subseteq I$ we write $\tau_j$ instead of $\tau_J$ and $\beta_j$ instead of $\beta_J$. In the diagram
\begin{equation*}
\xymatrix{
\Z[Y_{i,a}^{\pm 1}]_{i\in I, a\in \C^*}\ar[r]^-{\tau_j}\ar[d] &\Z[Y_{j,b}^{\pm}]_{j\in J,b\in\C^*}\otimes \Z[Z_{k,c}^{\pm}]_{k\in I\setminus J, c\in \C^*}\ar[d]\\
\Z[Y_{i,a}^{\pm 1}]_{i\in I, a\in \C^*}\ar[r]^-{\tau_j} & \Z[Y_{j,b}^{\pm}]_{j\in J,b\in\C^*}\otimes \Z[Z_{k,c}^{\pm}]_{k\in I\setminus J, c\in \C^*}
}
\label{tauj diagram}
\end{equation*}
the right vertical arrow be multiplication by $\beta_j(A_{j,c}^{-1})\otimes 1$; then the diagram commutes if and only if the left vertical arrow is multiplication by $A_{j,c}^{-1}$.
\end{enumerate}
In particular, the properties of $\tau_j$ imply the following.

\begin{lemma}
Let $V$ be a $\uqh g$-module, $m\in \chi_q(V)$ such that $\beta_j(m)\in \cal P_j^+$, and $v\in V_m$ such that $v$ is a highest $\ell$-weight vector for the action of $U_q(\hlie g_j)$. Then $m M\in \chi_q(V)$ for all $M\in \cal Q^-$ such that $\beta_j(m M)\in \chi_q(L(\beta_j(m)))$.  \qed
\label{sl2 char}
\end{lemma}

In what follows we often use the following lemma which is obtained from properties of $\beta_J$ described above together with \eqref{mon set} and the algebraic independence of $A_{j,a}$.	
\begin{lemma}
Let $m_+\in \cal P^+$ and $\{i_1,i_2,\ldots,i_N\} = I$. Let $k\in\{1,\dots,N-1\}$ and $M_j\in \cal Q_{i_j}^-$, $j=1,\dots,k $. Let $m = m_+ \prod_{j=1}^k M_j$. Then 
$x_{i_j,r}^{+} \cdot v=0$ for all $v\in L(m_+)_m$, $j=k+1,\dots,N$, and $r\in\Z$.\qed
\label{hws lemma} 
\end{lemma}

\section{Thin special $q$-characters and tame modules}\label{s:thin tame}
In this section we discuss tame and thin modules in more detail. In particular, we describe the results in the case of $\mathfrak {sl}_N$.

\subsection{First properties}
The main objects of this paper are the tame representations. In all cases they turn out to be also thin.
We start with simple remarks about tame and thin modules.

\begin{lemma}\label{r:tameresfactor}
The restriction of a tame module to any diagram subalgebra is tame. A subfactor of a tame module is tame. \qed
\end{lemma}

\begin{lemma}\label{r:thinresfactor}
A subfactor of a thin module is thin. \qed
\end{lemma}

Next we consider tensor products.

\begin{lemma}\label{l:simplethin} Let $m_1,m_2\in \mathcal P^+$.
If the module $L(m_1)\otimes L(m_2)$ is thin, then both $L(m_1)$ and $L(m_2)$ are thin. \qed
\end{lemma}

We also have the tame analogue.

\begin{lemma}\label{l:simpletame} Let $m_1,m_2\in \mathcal P^+$.
If the module $L(m_1)\otimes L(m_2)$ is tame, then both $L(m_1)$ and $L(m_2)$ are tame. 
\end{lemma}
\begin{proof}
Let $\{v_j\}_{j=1}^{n_1}$ be a basis of $L(m_1)$ and $\{w_j\}_{j=1}^{n_2}$ be a basis of $L(m_2)$ such that $\phi^\pm_i(u)$, $i\in I$, act diagonally in the basis $v_j\otimes w_k$. We also assume that $w_1$ is the highest weight vector and $v_{n_1}$ is the lowest weight vector. 

Let $f^\pm_i(u)$ and $g_i^\pm(u)$ be series in $u^{\pm1}$ defined by 
$\phi_i^\pm(u)w_1=f_i^\pm(u)w_1$ and $\phi_i^\pm(u)v_{n_1}=g_i^\pm(u)v_{n_1}$. Note that the series 
$f^\pm_i(u)$ and $g_i^\pm(u)$ are invertible.

By Lemma \ref{e:coprod}, we have
$$ f_i^\pm(u)\ (\phi_i^\pm(u)v_j)\otimes w_1=\phi_i^\pm(u)(v_j\otimes w_1)\in \C[[u^{\pm1}]]\ v_j\otimes w_1,$$
$i\in I$, $j=1,\dots,n_1,$ and 
$$ g_i^\pm(u)\ v_n\otimes (\phi_i^\pm(u) w_j)=\phi_i^\pm(u)(v_{n_1}\otimes w_j)\in \C[[u^{\pm1}]]\ v_{n_1}\otimes w_j,$$
$i\in I$, $j=1,\dots,n_2$. 

The lemma follows.
\end{proof}
The converse statements of Lemma \ref{l:simplethin} and Lemma \ref{l:simpletame} are false. 
For example, a two dimensional irreducible evaluation module of $U_q(\hlie {sl}_2)$ is tame and thin, but its tensor square is not tame and not thin.

The statements similar to Lemma \ref{l:simplethin} and Lemma \ref{l:simpletame}, for tensor products with more than two factors easily follow by induction.

\medskip
\subsection{The set $\cal X$ and a reduction of the problem.}\label{chi sec}
Throughout this paper we work only with simple finite-dimensional representations of $\uqh g$, for $\g$ of types $A$ and $B$.
We show that it is sufficient to study tame modules with dominant monomials belonging to the following subset of $\cal P^+$. 

Define the subset $\cal X\subset I\times \Z$ by
\begin{itemize}
\item[Type A] $\cal X:= \{(i,k)\in I\times \Z | i-k \equiv 1 \mod 2\}$.
\item[Type B] $\cal X:= \{(N, 2k +1)| k\in Z\}\bigsqcup \{(i,k)\in I\times \Z| i<N \ {\rm and}\ k\equiv 0 \mod 2\}$.
\end{itemize}

It is known that for all $m_+\in \cal P^+$, the module $L(m_+)$ can be written as a tensor product $L(m_+)=\otimes_{a\in\C^*} L(m_a)$ 
where each $L(m_a)$ is a module such that 
\begin{equation}\label{e:qcharX}
\chi_q(L(m_a))\in \Z[Y_{i,aq^k}^{\pm 1}]_{(i,k)\in \cal X}.
\end{equation} 

Clearly $L(m_+)$ is thin if and only if each $L(m_a)$ is thin. 
Therefore, by Lemma \ref{l:simpletame}, to classify all tame modules, it is sufficient to classify all tame modules for modules $L(m_a)$ satisfying (\ref{e:qcharX}) and to show that these modules are thin. It is done in Theorem \ref{mainthm2}. Then it follows  that $L(m_+)$ is tame if and only if all $L(m_a)$ are tame and, moreover, $L(m_+)$ is tame if and only if it is thin.

\medskip

Thus, we pick and fix a $c\in \C^*$ and to the rest of the paper we consider only representations 
whose $q$-characters lie in the subring $\Z[Y_{i,cq^k}^{\pm 1}]_{(i,k)\in \cal X}$.

Recall that we have $r_i=1$ for type $A_N$, and in type $B_N$ we have $r_i=2$, $i<N$, $r_N=1$. We define
$$\cal W:= \{(i,k)| (i,k-r_i)\in \cal X\}.$$
Then we have $\cal M(L(m_+))\subseteq m_+\Z[A_{i,cq^k}^{-1}]_{(i,k)\in \cal W}$, for all $m_+\in \Z[Y_{i,cq^k}]_{(i,k)\in \cal X}$. This is a refinement of (\ref{mon set}).

Henceforth, by an abuse of notation, we  write
$$Y_{i,k}:= Y_{i,cq^k},\quad A_{i,k}:= A_{i,cq^k}, \quad u_{i,k}:= u_{i,cq^k},$$
for all $(i,k)\in \cal X$. 

We denote by $\mathcal P_{\cal X}$ (resp. $\cal P_{\cal X}^{+}$) the subgroup (resp. submonoid) of $\cal P$ generated by $Y_{i,k},\ (i,k)\in \cal X$.

Given $m\in \mathcal P_{\cal X}^+$, we always write $m = \prod_{t=1}^T Y_{i_t,k_t}$  in such a way that 
$k_{t+1}\geq k_t$, and $i_{t+1}\geq i_t$ whenever $k_{t+1}=k_t$. We denote  the ordered sequence $(i_t,k_t)_{1\leq t\leq T}$ by $\cal X(m)$.	

\subsection{Thin $U_q(\hlie {sl}_2)$-modules}

In this subsection we fix $\g = \mathfrak {sl}_2$. 

The set $S_k(b):= \{bq^{-k+1}, bq^{-k+3},\ldots, bq^{k-1}\}$ is called the \emph{$q$-string} of length $k\in \Z_{\geq 0}$ centered at $b\in \C^*$. Two $q$-strings are said to be in \emph{general position} if one contains the other or their union is not a $q$-string. Let $m_b^{(k)}:= Y_{1,bq^{-k+1}}Y_{1,bq^{-k+3}}\ldots Y_{1,bq^{k-1}}$. For each $m_+\in \cal P$ there is a unique multiset of $q$-strings in pairwise general position, denoted by $\gbr S_q(m_+)$, such that
\begin{equation}\label{e:qstrings}
\gbr S_q(m_+) = \{S_{k_t}(b_t)| 1\leq t\leq n\}\quad {\rm and} \quad m_+ = \prod_{t=1}^{n}m_{b_t}^{(k_t)}. 
\end{equation} 
Moreover, 
$$L(m_+) \cong \bigotimes_{t=1}^{n} L(m_{b_t}^{(k_t)}),$$
where $L(m_b^{(k)})$ is an evaluation module, and
\begin{equation}\label{e:qchev}
\chi_q(L(m_b^{(k)})) = m_b^{(k)}\left(1+ \sum_{t=0}^{k-1}A_{1,bq^k}^{-1}A_{1,bq^{k-2}}^{-1} \ldots A_{1,bq^{k-2t}}^{-1}\right),
\end{equation}
see \cite{CP94a}.

The module $L(m_+)$ is a thin simple finite-dimensional representation if and only if each two $q$-strings in $\gbr S_q(m_+)$ are  pairwise disjoint, see \cite{NT98}. Moreover, we have the following. 

\begin{lemma}[\cite{NT98}] \label{thin sl2} Let $V$ be a finite-dimensional simple $U_q(\hlie{sl}_2)$-module. Then the $V$ is tame if and only if $V$ is thin. All thin modules are special. \qed
\end{lemma} 

The following useful lemma describes the monomials which can be found in a simple thin module.

\begin{lemma}\cite[Lemma 3.1]{MY12a}\label{sl2 monomials}
Let $m\in \cal P$. There exists $m_+\in \cal P^+$ such that $L(m_+)$ is thin and $m\in \cal M(L(m_+))$ if and only if, for all $b\in \C^*$, $|u_{1,b}(m)|\leq 1$ and $u_{1,b}(m) - u_{1,bq^2}(m) \neq 2$. Moreover, $mA_{1,bq}^{-1}$ is a monomial of $L(m_+)$ iff $u_{1,b}(m) =1$ and $u_{1,bq^2}(m) = 0$. \qed
\label{mon sl2 condition}
\end{lemma}

\subsection{Sufficient criteria for determining the $q$-character of a thin special module.}

Given a set of monomials $\cal M$, \cite[Theorem 3.4]{MY12a} gives sufficient conditions to guarantee it corresponds to the $q$-character of a thin special finite-dimensional $\uqh g$-module. 

\begin{thm}\cite{MY12a}
Let $m_+\in \cal P^+$. Suppose that $\cal M\subseteq \cal P$ is a finite set of distinct monomials such that 
\begin{enumerate}
\item $\{m_+\} = \cal P^+\cap \cal M$,
\item for all $m\in \cal M$ and all $(i,a)\in I\times \C^*$, if $mA^{-1}_{i,a}\notin \cal M$ then $mA^{-1}_{i,a}A_{j,b}\notin \cal M$ unless $(j,b)=(i,a)$,
\item for all $m\in \cal M$ and all $i\in I$ there exists $M\in \cal M$, $i$-dominant, such that
$$\chi_q(L(\beta_i(M))) = \sum_{m'\in m\Z[A_{i,a}^{\pm 1}]_{a\in \C^*}\cap \cal M}\beta_i(m').$$
Then
$$\chi_q(L(m_+)) = \sum_{m\in \cal M}m,$$
and, in particular, $L(m_+)$ is thin and special.
\end{enumerate} 
\qed
\label{thincriteria}
\end{thm}
We use this theorem below to compute $q$-characters of all thin irreducible modules in types $A$ and $B$.

\subsection{The $\lie {sl}_N$ case.}

In this subsection let $\g = \lie{sl}_N$. We recall the well-known results about tame simple representations of $U_q(\hlie g)$.

A point $(i',k')\in \cal X$ is said to be in \emph{snake position} to $(i,k)\in \cal X$ if 
$$k'-k \geq 2 + |i'-i|.$$
A sequence of points $(i_t,k_t)\in \cal X$, $1\leq t\leq T$, $T\in \Z_{\geq 0}$, is called a \emph{snake} if $(i_t,k_t)$ is in snake position to $(i_{t-1},k_{t-1})$ for all $2\leq t \leq T$. The following is essentially a result of 
\cite[Theorem 4.1]{NT98}, see also \cite[Lemma 4.7]{FM01} and  \cite[Theorem 6.1]{MY12a}. 

\begin{thm}\cite{NT98}
Let $\g = \lie{sl}_N$. Let $m_+\in\mathcal P_{\cal X}^+$ and let $\cal X(m_+)$ be a snake. Then $L(m_+)$ is special, thin and therefore tame.
\label{snakethin}\qed
\end{thm}

This theorem is proved by an explicit computation of the $q$-character. The converse statement is also essentially a result of \cite[Theorem 4.1]{NT98}. We give a proof to illustrate the methods we use in the proof of Theorem \ref{mainthm2}.

\begin{thm}\cite{NT98}
Let $\g = \lie{sl}_N$. Let $m_+ \in\mathcal P^+_{\cal X}$.
Then $L(m_+)$ is a tame representation if and only if $\cal X(m_+)$ is a snake.
\label{NT}
\end{thm}
\begin{proof} The \emph{if} part follows from Theorem \ref{snakethin}. For the \emph{only if} part we proceed by induction on $N$. The case of $N=2$ follows from Lemma \ref{thin sl2}. 

Write $m_+ = \prod_{t=1}^T Y_{i_t,k_t}$, where $T\in \Z_{\geq 0}$,  $(i_t,k_t)\in \cal X$, $1\leq t\leq T$.
For brevity, write $V= L(m_+)$. Let $J=\{1,2,\ldots, N-1\}$ and $K=\{2,3,\ldots, N\}$ be subsets of $I$.
By Lemma \ref{r:tameresfactor}, the $U_q(\hlie g_J)$-module $L(\beta_J(m_+))$ is tame. Note that $U_q(\hlie g_J)$ is isomorphic to $U_q(\hat{\mathfrak{sl}}_{N-1})$. Therefore by the induction hypothesis, 
$$k_{t+1} - k_t \geq 2+ |i_{t+1} -i_t|,$$
whenever $N\notin \{i_t,i_{t+1}\}$. 

Similar arguments applied to $\res_K V$ show that 
$$k_{t+1} - k_t \geq 2 + |i_{t+1}-i_t|,$$
whenever $1\notin\{i_t, i_{t+1}\}$.

Therefore it remains to consider the case of $\{i_t, i_{t+1}\} = \{1,N\}$. Suppose by contradiction that
$0\leq k_{t+1}-k_t < 2+|N-1|$. By definition of the set $\cal X$ it is equivalent to $k_{t+1}-k_{t} \leq N-1$. By Lemma \ref{sl2 char}, we have
$$m =m_+ A_{i_{t}, k_{t}+1}^{-1}\in \chi_q(V).$$
Explicitly, 
$$m = \left\{\begin{array}{ll}
m_+ Y_{1,k_t}^{-1}Y_{1,k_t+2}^{-1}Y_{2, k_t+1} & {\rm if}\ i_t=1 \ {\rm and}\ i_{t+1}= N,\\
m_+ Y_{N,k_t}^{-1}Y_{N,k_t+2}^{-1}Y_{N-1, k_t+1} & {\rm if}\ i_t=N \ {\rm and}\ i_{t+1}= 1.\\
\end{array}\right.
$$
In the first case, by Lemma \ref{hws lemma}, $L(\beta_J(m))$ is a subfactor of $\res_JV$.
However,  
$$k_{t+1} - (k_{t}+1)\leq (N-1)-1 = |N-2|,$$
yields a contradiction with the inductive hypothesis. 
A similar argument proves that the second case cannot hold either.
\end{proof}
Since all simple tame $U_q(\hlie {sl}_N)$-modules are special, the following is immediate from the previous theorem and Lemma \ref{r:tameresfactor}.
 
\begin{cor}\label{c:tameslN}
Let $V$ be a tame $U_q(\hlie{sl}_N)$-module. Then, for each dominant monomial $m\in \cal M(V)$, $L(m)$ is a subfactor of $V$.\qed
\end{cor}

\section{Extended snake modules}\label{s:extsnake}

Assume $\g$ to be of type $B_N$.
In this section we recall the definition of snakes for type $B$ and extend such definition to include all cases where the $q$-character formula of Theorem 6.1 in \cite{MY12a} pertains, see Theorem \ref{mainthm1} and Remark \ref{extended def rem}. For the remainder of the paper $\delta_{ij}$ will donote the Kronecker delta.

\subsection{Definition}

Let $(i,k)\in \cal X$. Let us say that a point $(i',k')\in \cal X$ is in \emph{extended snake position} with respect to $(i,k)$ if at least one of the following conditions is true.
\begin{enumerate}
\item $k'-k \geq 4+ 2|i'-i| - \delta_{Ni} - \delta_{Ni'}$ \quad and\quad $k'-k \equiv 2(i' -i) - \delta_{Ni} - \delta_{Ni'}\mod 4$, 
\item $k' -k\geq 2N + 2 + 2|N-i-i'|- \delta_{Ni} - \delta_{Ni'}$, 
\end{enumerate}
A point $(i',k')\in \cal X$ is in \emph{snake position} with respect to $(i,k)$ if the condition (i) holds, see \cite{MY12a}.
A point $(i',k')\in \cal X$ is said to be in \emph{minimal snake position} to $(i,k)$ if $k'-k$ is equal to the lower bound in (i).

Following \cite{MY12a}, we draw the images of points in $\cal X$ under the injective map $\iota: \cal X \rightarrow \Z\times \Z$ defined as follows
\begin{equation}
\iota: (i,k)\mapsto \left\{\begin{array}{ll}
(2i, k) & i<N \ {\rm and}\ 2N+k-2i \equiv 2 \mod 4,\\
(4N-2-2i, k) & i<N \ {\rm and }\ 2N+k-2i \equiv 0 \mod 4,\\
(2N-1, k) & i=N.
\end{array}\right.
\end{equation}

\begin{figure} \caption{\label{snakeposfig} Each of the points $\scriptscriptstyle\blacklozenge$, $\scriptscriptstyle\square$ and $\scriptscriptstyle\blacksquare$ are, respectively, in extended snake, snake and minimal snake position to the point marked $\circ$.}
\begin{equation*}
\begin{tikzpicture}[scale=.3,yscale=-1]
\draw[help lines] (0,0) grid (18,23);
\draw (2,-1) -- (8,-1); \draw (10,-1) -- (16,-1);
\draw[double,->] (8,-1) -- (8.8,-1); \draw[double,->] (10,-1) -- (9.2,-1);
\filldraw[fill=white] (9,-1) circle (2mm) node[above=1mm] {$\scriptstyle 5$};
\foreach \x in {1,2,3,4} {
\filldraw[fill=white] (2*\x,-1) circle (2mm) node[above=1mm] {$\scriptstyle\x$}; 
\filldraw[fill=white] (2*9-2*\x,-1) circle (2mm) node[above=1mm] {$\scriptstyle\x$}; }
\foreach \y in {0,2,4,6,8,10,12,14,16,18,20,22} {\node at (-1,\y) {$\scriptstyle\y$};}
\draw[thick] (6,2) circle (2mm);
\foreach \x/\y in {2/10,2/14,2/18,2/22,4/8,6/6,6/10,8/8,8/12,9/9,9/13,4/12,4/16, 4/20,6/14,6/18,6/22,8/16,8/20,9/17,9/21} 
{\node[regular polygon, regular polygon sides=4,draw,fill=white,inner sep=.3mm] at (\x,\y) {};}
\foreach \x/\y in {2/10,4/8,6/6,8/8,9/9} 
{\node[regular polygon, regular polygon sides=4,draw,fill=black,inner sep=.3mm] at (\x,\y) {};}
\foreach \x/\y in {16/20, 16/16, 14/18, 14/14,14/22, 12/20, 12/16, 10/22, 10/18, 9/23, 9/19}
{\node[shape border rotate = 45, regular polygon, regular polygon sides = 4, draw, fill = black, inner sep = .32mm] at (\x,\y) {};}
\end{tikzpicture}
\begin{tikzpicture}[scale=.3,yscale=-1]
\draw[help lines] (0,0) grid (18,23);
\draw (2,-1) -- (8,-1); \draw (10,-1) -- (16,-1);
\draw[double,->] (8,-1) -- (8.8,-1); \draw[double,->] (10,-1) -- (9.2,-1);
\filldraw[fill=white] (9,-1) circle (2mm) node[above=1mm] {$\scriptstyle 5$};
\foreach \x in {1,2,3,4} {
\filldraw[fill=white] (2*\x,-1) circle (2mm) node[above=1mm] {$\scriptstyle\x$}; 
\filldraw[fill=white] (2*9-2*\x,-1) circle (2mm) node[above=1mm] {$\scriptstyle\x$}; }
\foreach \y in {0,2,4,6,8,10,12,14,16,18,20,22} {\node at (-1,\y) {$\scriptstyle\y$};}
\draw[thick] (9,1) circle (2mm);
\foreach \x/\y in {9/3,9/7,9/11,9/15,9/19, 9/23, 10/6,10/10,10/14,10/18,10/22, 12/8, 12/12,12/16, 12/20, 14/10, 14/14, 14/18,14/22, 16/12,16/16, 16/20} 
{\node[regular polygon, regular polygon sides=4,draw,fill=white,inner sep=.3mm] at (\x,\y) {};}
\foreach \x/\y in {9/3, 10/6, 12/8, 14/10, 16/12} 
{\node[regular polygon, regular polygon sides=4,draw,fill=black,inner sep=.3mm] at (\x,\y) {};}
\foreach \x/\y in {9/21, 8/20, 6/18, 6/22, 4/16, 4/20, 2/14, 2/18, 2/22}
{\node[shape border rotate = 45, regular polygon, regular polygon sides = 4, draw, fill = black, inner sep = .32mm] at (\x,\y) {};}
\end{tikzpicture}
\end{equation*}
\end{figure}
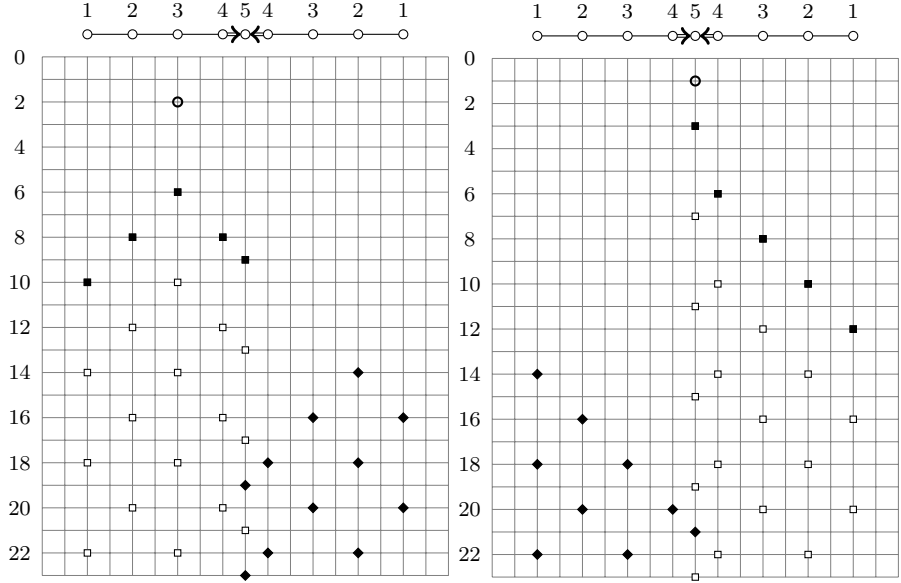

Let $(i_t,k_t)$, $1\leq t\leq T$, $T\in \Z_{\geq 1}$, be a sequence of points in $\cal X$. The sequence $(i_t,k_t)_{1\leq t\leq T}$ is said to be an \emph{extended snake} (resp. \emph{snake}) if $(i_t,k_t)$ is in extended snake position (resp. snake position) to $(i_{t-1},k_{t-1})$ for all $2\leq t\leq T$. We call the simple module $L(m)$ an \emph{extended snake module} (resp. \emph{snake module}) if $\cal X(m)$ is an extended snake (resp. snake).

Recall that the {\it minimal snake} is the snake such that $(i_t,k_t)$ is in minimal snake position to $(i_{t-1},k_{t-1})$ for all $2\leq t\leq T$ and the corresponding simple module $L(m)$ is called \emph{minimal snake module}.
Recall further that the {\it minimal affinizations} are minimal snake modules, where the sequence of $i_t$, $t=1,\dots,T$, is monotone. Finally, recall that the
 {\it Kirillov-Reshetikhin modules} are minimal affinizations where the sequence $i_t$, $t=1,\dots,T$, is constant. Therefore
we have the following families of representations, in order of increasing generality:
\begin{center} \text{KR modules $\subset$ minimal affinizations $\subset$ minimal snakes $\subset$ snakes $\subset$ extended snakes.}
\end{center}

\begin{rem}\label{not transitive} The extended snake position is not a transitive concept. Namely, if $(i',k')$ is in extended snake position to $(i,k)$ and $(i'',k'')$ is in extended snake position to $(i',k')$,  then it does not follow in general that $(i'',k'')$ is in  snake position to $(i,k)$. However, in the cases when not, we necessarily have $i'=N$, $i+i''>N-2$, $(i',k')$ is in snake position to $(i,k)$ and $(i'',k'')$ is in snake position to $(i',k')$. \end{rem}

\subsection{Paths, corners, raising and lowering moves}

A \emph{path} is a finite sequence of points in the plane $\R^2$. We write $(i,k)\in p$ if $(i,k)$ is a point of the path $p$. In our diagrams, we connect consecutive points of the path by line segments, for illustrative purposes only. For each $(i,k)\in \cal X$ we define a set $\scr P_{i,k}$ of paths. Pick and fix an $\epsilon$, $1/2>\epsilon>0$. Define $\scr P_{N,k}$ for all $k\in 2\Z +1$ in the following way
\begin{itemize}
\item For all $k\equiv 3 \mod 4$,
\begin{align*}
\scr P_{N,k} & :=\{\left((0,y_0), (2,y_1), \ldots, (2N-4, y_{N-2}), (2N-2, y_{N-1}),(2N-1, y_N)\right)|\\
& \qquad y_0 = k+2N-1, y_{j+1}-y_j \in \{2,-2\} \ \forall\  0\leq j\leq N-2 \\
& \qquad {\rm and}\ y_N - y_{N-1} \in \{1+\epsilon, -1 -\epsilon\}\}.
\end{align*}
\item For all $k\equiv 1 \mod 4$,
\begin{align*}
\scr P_{N,k} & :=\{\left((4N-2,y_0), (4N-4,y_1), \ldots, (2N+2, y_{N-2}), (2N, y_{N-1}),(2N-1, y_N)\right)|\\
& \qquad y_0 = k+2N-1, y_{j+1}-y_j \in \{2,-2\} \ \forall\ 0\leq j\leq N-2 \\
& \qquad {\rm and}\ y_N - y_{N-1} \in \{1+\epsilon, -1 -\epsilon\}\}.
\end{align*}
\end{itemize}
Next, $\scr P_{i,k}$ is defined, for all $(i,k)\in \cal X$, as follows.
\begin{align*}
\scr P_{i,k} & :=\{\left(a_0,a_1,\ldots, a_N,\ol{a_{N}}, \ldots, \ol{a_1},\ol{a_0}\right)|\\
& \qquad(a_0,a_1\ldots, a_N)\in \scr P_{N,k-(2N-2i-1)}, (\ol{a_0}, \ol{a_1}, \ldots, \ol{a_N})\in \scr P_{N,k+(2N-2i-1)},\\
& \qquad {\rm and}\ a_N - \ol{a_N} = (0,y) \ {\rm where}\ y>0\}.
\end{align*}
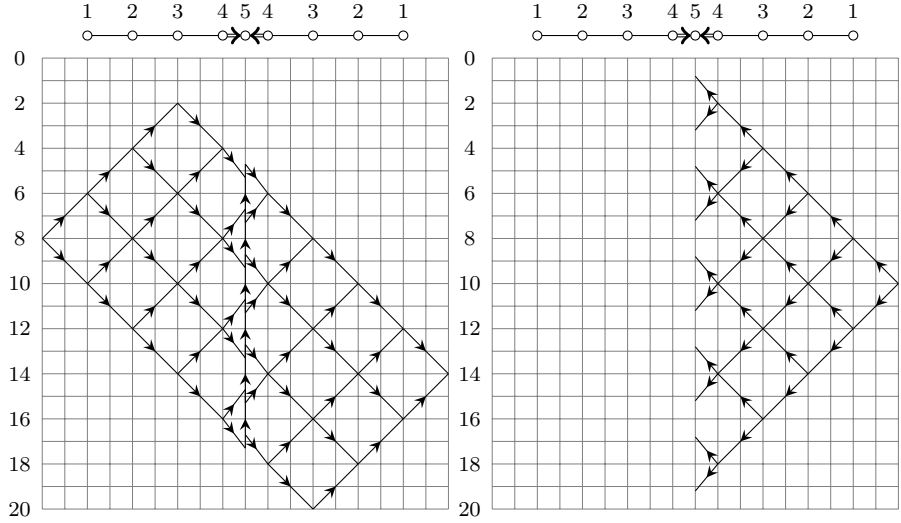
\begin{figure} \caption{\label{pathsfig} In type $B_5$ illustration of the paths in $\scr P_{3,2}$ (left) and $\scr P_{5,1}$ (right).}
\begin{equation*}
\begin{tikzpicture}[scale=.3,yscale=-1,decoration={
markings,
mark=at position .5 with {\arrow[black,line width=.3mm]{stealth}}}]
\draw[help lines] (0,0) grid (18,20);
\draw (2,-1) -- (8,-1); \draw (10,-1) -- (16,-1);
\draw[double,->] (8,-1) -- (8.8,-1); \draw[double,->] (10,-1) -- (9.2,-1);
\filldraw[fill=white] (9,-1) circle (2mm) node[above=1mm] {$\scriptstyle 5$};
\foreach \x in {1,2,3,4} {
\filldraw[fill=white] (2*\x,-1) circle (2mm) node[above=1mm] {$\scriptstyle\x$}; 
\filldraw[fill=white] (2*9-2*\x,-1) circle (2mm) node[above=1mm] {$\scriptstyle\x$}; }
\foreach \y in {0,2,4,6,8,10,12,14,16,18,20} {\node at (-1,\y) {$\scriptstyle\y$};}
\foreach \x/\y in {0/8,2/6, 2/10,4/4, 6/6, 6/10, 6/14, 4/8, 4/12, 10/10, 10/14,10/18,12/12, 12/16, 14/14} {\draw[postaction={decorate}] (\x,\y) -- ++(2,2); \draw (\x,\y)[postaction={decorate}] --++(2,-2);}
\foreach \x/\y in {6/2, 10/6, 12/8, 14/10, 16/12} {\draw[postaction={decorate}] (\x,\y) -- ++(2,2);}
\foreach \x/\y in {12/20, 14/18, 16/16} {\draw[postaction={decorate}] (\x,\y) -- ++(2,-2);}
\foreach \x/\y in {8/4, 8/8, 8/12, 8/16} {\draw[postaction={decorate}] (\x,\y) -- ++(1,1.3);
\draw[postaction={decorate}] (\x,\y)++(1,0.7) -- ++(1,1.3); }
\foreach \x/\y in {8/8, 8/12, 8/16} {\draw[postaction={decorate}] (\x,\y) -- ++(1,-1.3);
\draw[postaction={decorate}] (\x,\y)++(1,-0.7) -- ++(1,-1.3); }
\foreach \x/\y in {9/17, 9/15, 9/13, 9/11, 9/9, 9/7} {\draw[postaction={decorate}] (\x,\y)++(0,.3) -- ++(0,-2.6); }
\end{tikzpicture}
\begin{tikzpicture}[scale=.3,yscale=-1,decoration={
markings,
mark=at position .5 with {\arrow[black,line width=.3mm]{stealth}}}
]
\draw[help lines] (0,0) grid (18,20);
\draw (2,-1) -- (8,-1); \draw (10,-1) -- (16,-1);
\draw[double,->] (8,-1) -- (8.8,-1); \draw[double,->] (10,-1) -- (9.2,-1);
\filldraw[fill=white] (9,-1) circle (2mm) node[above=1mm] {$\scriptstyle 5$};
\foreach \x in {1,2,3,4} {
\filldraw[fill=white] (2*\x,-1) circle (2mm) node[above=1mm] {$\scriptstyle\x$}; 
\filldraw[fill=white] (2*9-2*\x,-1) circle (2mm) node[above=1mm] {$\scriptstyle\x$}; }
\foreach \y in {0,2,4,6,8,10,12,14,16,18,20} {\node at (-1,\y) {$\scriptstyle\y$};}
\foreach \x/\y in {18/10, 16/8, 16/12, 14/10, 14/6, 14/14, 12/4, 12/8, 12/12, 12/16} {\draw[postaction={decorate}] (\x,\y) -- ++(-2,2); \draw (\x,\y)[postaction={decorate}] --++(-2,-2);}
\foreach \y in {2,6,10,14,18}  {\draw[postaction={decorate}] (10,\y) -- ++(-1,1.2); \draw[postaction={decorate}] (10,\y) --++(-1,-1.2);}
\end{tikzpicture}
\end{equation*}
\end{figure}

For all $(i,k)\in \cal X$, we define the sets of upper and lower \emph{corners} $C_{p,\pm}$ of a path $p = ((j_r,l_r))\in \scr P_{i,k}$ as follows:
\begin{align*}
C_{p,+} & :=\iota^{-1}\{(j_r,l_r)\in p | j_r\notin \{0,2N-1,4N-2\}, \ l_{r-1}>l_r,\ l_{r+1}>l_r\}\\
& \qquad \sqcup \{(N,l)\in \cal X | (2N-1, l-\epsilon)\in p \ {\rm and}\ (2N-1,l+\epsilon)\notin p\},
\end{align*}
\begin{align*}
C_{p,-} & :=\iota^{-1}\{(j_r,l_r)\in p | j_r\notin \{0,2N-1,4N-2\}, \ l_{r-1}<l_r,\ l_{r+1}<l_r\}\\
& \qquad \sqcup \{(N,l)\in \cal X | (2N-1, l+\epsilon)\in p \ {\rm and}\ (2N-1,l-\epsilon)\notin p\}.
\end{align*}

For all $(i,k)\in \cal X$ we define the \emph{highest path} of $\scr P_{i,k}$, denoted by $p_{i,k}^+$, to be the unique element of $\scr P_{i,k}$ with no lower corners. Analogously, we define $p_{i,k}^-$, the \emph{lowest path}, to be the unique element of $\scr P_{i,k}$ with no upper corners. 

We define a map $\m$ sending paths to monomials, as follows:
\begin{equation}
\begin{array}{rrcl}
\m:& \displaystyle{\bigsqcup_{(i,k)\in \cal X}\scr P_{i,k}} & \rightarrow& \Z\left[Y_{j,l}^{\pm 1}\right]_{(j,l)\in \cal X}\\
& p&\mapsto& \m(p):= \displaystyle{\prod_{(j,l)\in C_{p,+}}Y_{j,l}\prod_{(j,l)\in C_{p,-}}Y_{j,l}^{-1}}.
\end{array}
\label{monfunction}
\end{equation}

Observe that, for each $(i,k)\in \cal X$ and $p\in \scr P_{i,k}$, $\m(p) \in \cal P^+$ (resp. $\m(p)^{-1}\in \cal P^+$) if and only if $p =p_{i,k}^+$ (resp. $p = p_{i,k}^{-}$).

Let $(i,k)\in \cal X$ and $(j,l)\in \cal W$. We say a path $p\in \scr P_{i,k}$ can be \emph{lowered} at $(j,l)$ if $(j,l-r_j)\in C_{p,+}$ and $(j,l+r_j)\notin C_{p,+}$. If $(i,k)$ can be lowered at $(j,l)$ we define a \emph{lowering move} at $(j,l)$, resulting in another path in $\scr P_{i,k}$ which we write as $p\scr A_{j,l}^{-1}$ and which is defined to be the unique path such that $\m(p\scr A_{j,l}^{-1}) = \m(p)A_{j,l}^{-1}$. A detailed case-by-case description of these moves can be found in \cite{MY12a}, Section 5.

\subsection{Non-overlapping paths}

A path $p$ is said to be \emph{strictly above} a path $p'$ if 
$$(x,y)\in p\quad {\rm and} \quad (x,z)\in p' \Rightarrow y<z.$$ If a path $p$ is strictly above a path $p'$, then we also say
$p'$ \emph{strictly bellow} $p$.
A $T$-tuple of paths $(p_1,\ldots, p_T)$ is said \emph{non-overlapping} if $p_s$ is strictly above $p_t$ for all $s<t$. Otherwise, for some $s<t$ there exist $(x,y)\in p_s$ and $(x,z)\in p_t$ such that $y\geq z$, and we say $p_s$ \emph{overlaps} $p_t$ in column $x$, see Figure 3.

\begin{figure}\caption{Illustration of the definition of overlapping paths in type $B_3$. By Theorem \ref{mainthm1} bellow, $L(Y_{3,1}Y_{3,3})$ contains the monomial $(Y_{1,4}Y_{3,7}^{-1})(Y_{3,9}^{-1}Y_{2,8}Y_{1,10}^{-1})$ (left) but {\emph not} the monomial $Y_{1,8}^{-1}Y_{2,6}Y_{2,8}^{-1}Y_{1,6} = (Y_{1,8}^{-1}Y_{2,6}Y_{3,7}^{-1})(Y_{3,7}Y_{2,8}^{-1}Y_{1,6})$(right).}
\begin{equation*}
\begin{tikzpicture}[scale=.4,yscale=-1,decoration={
markings,
mark=at position .5 with {\arrow[black,line width=.3mm]{stealth}}}]
\draw[help lines] (0,0) grid (10,12);
\draw (2,-1) -- (4,-1); \draw (6,-1) -- (8,-1);
\draw[double,->] (4,-1) -- (4.8,-1); \draw[double,->] (6,-1) -- (5.2,-1);
\filldraw[fill=white] (5,-1) circle (2mm) node[above=1mm] {$\scriptstyle 3$};
\foreach \x in {1,2} {
\filldraw[fill=white] (2*\x,-1) circle (2mm) node[above=1mm] {$\scriptstyle\x$}; 
\filldraw[fill=white] (2*5-2*\x,-1) circle (2mm) node[above=1mm] {$\scriptstyle\x$}; }
\foreach \y in {0,2,4,6,8,10} {\node at (-1,\y) {$\scriptstyle\y$};}
\begin{scope}[every node/.style={minimum size=.1cm,inner sep=0mm,fill,circle}]
\draw[thick]  (0,6)node {}  -- (2,4)node {}  -- (4,6)node {} -- (5,7.2)node{};
 \draw[thick] (5, 9.2)node{} -- (6, 8)node{} -- (8, 10)node{} -- (10, 8)node{};
\end{scope}
\end{tikzpicture}
\begin{tikzpicture}[scale=.4,yscale=-1,decoration={
markings,
mark=at position .5 with {\arrow[black,line width=.3mm]{stealth}}}]
\draw[help lines] (0,0) grid (10,12);
\draw (2,-1) -- (4,-1); \draw (6,-1) -- (8,-1);
\draw[double,->] (4,-1) -- (4.8,-1); \draw[double,->] (6,-1) -- (5.2,-1);
\filldraw[fill=white] (5,-1) circle (2mm) node[above=1mm] {$\scriptstyle 3$};
\foreach \x in {1,2} {
\filldraw[fill=white] (2*\x,-1) circle (2mm) node[above=1mm] {$\scriptstyle\x$}; 
\filldraw[fill=white] (2*5-2*\x,-1) circle (2mm) node[above=1mm] {$\scriptstyle\x$}; }
\foreach \y in {0,2,4,6,8,10} {\node at (-1,\y) {$\scriptstyle\y$};}
\begin{scope}[every node/.style={minimum size=.1cm,inner sep=0mm,fill,circle}]
\draw[thick]  (0,6)node {}  -- (2,8)node {}  -- (4,6)node {} -- (5,7.2)node{};
 \draw[thick] (5, 6.8)node{} -- (6, 8)node{} -- (8, 6)node{} -- (10, 8)node{};
\end{scope}
\end{tikzpicture}
\end{equation*}
\label{overlapfig}
\end{figure}
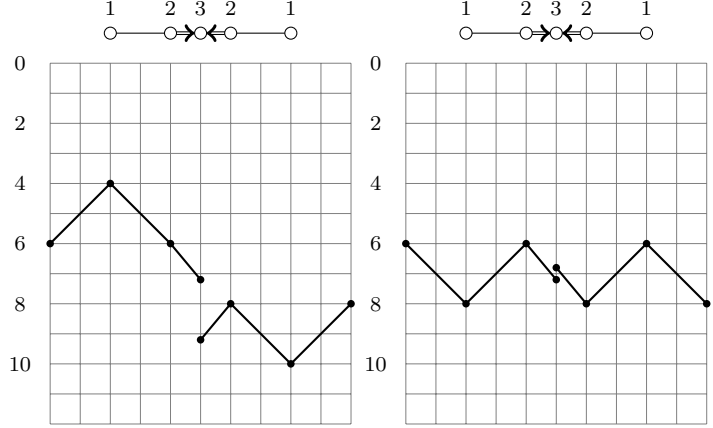

Let $(i_t,k_t)\in \cal X, 1\leq t\leq T$, $T\in\Z_{\geq 1}$, be an extended snake. Define
$$\nops := \{\ps | p_{t}\in \scr P_{i_t,k_t}, 1\leq t \leq T, \ps \ \textrm{is non-overlapping}\}.$$

\begin{lemma} If $(i_t,k_t)\in \cal X, 1\leq t\leq T$, $T\in \Z_{\geq 1}$, is an extended snake, then
$(p_{i_1,k_1}^+,\dots, p_{i_T,k_T}^+)\in\nops$ and $(p_{i_1,k_1}^-,\dots, p_{i_T,k_T}^-)\in\nops$.
\end{lemma}
\begin{proof}
We argue about highest weight paths. Lowest weight paths are treated similarly. 

If 
\begin{equation}\label{parity1}
k'-k \equiv 2 + 2(i'-i)-\delta_{Ni'} - \delta_{Ni} \mod 4,
\end{equation}
then $p_{i,k}^+$ is strictly above $p_{i',k'}^+$ because of the inequality
\begin{equation}\label{e:domnonoverlap1}
k'-k \geq 4N + 2 - 2i -2i'-\delta_{iN} - \delta_{i'N}.
\end{equation} 

And if
\begin{equation}\label{parity0}
k'-k \equiv 2(i'-i)-\delta_{Ni'} - \delta_{Ni}\mod 4,
\end{equation}
then $p_{i,k}^+$ is strictly above $p_{i',k'}^+$ because of the inequality
\begin{equation}\label{e:domnonoverlap}
k'-k \geq 4+ 2|i' -i|-\delta_{iN} - \delta_{i'N}.
\end{equation}

Therefore $p_{i_s,k_s}^+$ and $p_{i_{s+1},k_{s+1}}^+$ are non-overlapping for $s=1,\dots,T-1$.
The check that non-adjacent paths are non-overlapping, see Remark \ref{not transitive}, is also straightforward.
\end{proof}

\begin{lemma}
Let $(i_t,k_t)\in \cal X, 1\leq t \leq T$, $T\in \Z_{\geq 1}$, be an extended snake and $\ps\in \nops$. Suppose $(i,k)\in C_{p_t,\pm}$ for some $1\leq t \leq T$. Then
\begin{enumerate}
\item $(i,k)\notin C_{p_s,\pm}$ for any $s\neq t, \ 1\leq s \leq T$, and
\item if $(i,k)\in C_{p_s,\mp}$ for some $s$, $1\leq s \leq T$, then $s=t\pm 1$ and $i=N$.
\end{enumerate}
\label{overlapcond}
\end{lemma}
\begin{proof}
The lemma follows from the definitions of non-overlapping paths. Examples of (ii) are shown in Figure \ref{figoverlap}.
\end{proof}

\begin{figure}\caption{Illustration of Lemma \ref{overlapcond}. By Theorem \ref{mainthm1} bellow, $L(Y_{3,0}Y_{2,6})$ contains the monomial $Y_{3,12}^{-1}Y_{4,17}^{-1}Y_{3,10}=(Y_{3,12}^{-1}Y_{4,11})(Y_{4,17}^{-1}Y_{4,11}Y_{3,10})$ (left) and $L(Y_{4,1}Y_{4,3})$ contains the monomial $Y_{2,6}Y_{2,12}^{-1}=(Y_{2,6}Y_{4,9}^{-1})(Y_{4,9}Y_{2,12}^{-1})$(right).}
\begin{equation*}
\begin{tikzpicture}[scale=.4,yscale=-1,decoration={
markings,
mark=at position .5 with {\arrow[black,line width=.3mm]{stealth}}}]
\draw[help lines] (0,0) grid (14,18);
\draw (2,-1) -- (6,-1); \draw (8,-1) -- (12,-1);
\draw[double,->] (6,-1) -- (6.8,-1); \draw[double,->] (8,-1) -- (7.2,-1);
\filldraw[fill=white] (7,-1) circle (2mm) node[above=1mm] {$\scriptstyle 4$};
\foreach \x in {1,2,3} {
\filldraw[fill=white] (2*\x,-1) circle (2mm) node[above=1mm] {$\scriptstyle\x$}; 
\filldraw[fill=white] (2*7-2*\x,-1) circle (2mm) node[above=1mm] {$\scriptstyle\x$}; }
\foreach \y in {0,2,4,6,8,10,12,14,16,18} {\node at (-1,\y) {$\scriptstyle\y$};}
\begin{scope}[every node/.style={minimum size=.1cm,inner sep=0mm,fill,circle}]
\draw[thick]  (0,6)node {}  -- (2,8)node {}  -- (4,10)node {} -- (6, 12)node{} -- (7, 10.8)node{} -- (7, 0.8)node{} -- (8, 2)node{} -- (10, 4)node{}-- (12, 6)node{} -- (14, 8)node{};
 \draw[thick] (0,10)node{} -- (2, 12)node{} -- (4, 14)node{} -- (6, 16)node{} -- (7, 17.2)node{} -- (7, 11.2)node{} -- (8, 10)node{} -- (10, 12)node{} -- (12, 14)node{} -- (14, 16)node{};
\end{scope}
\end{tikzpicture}
\begin{tikzpicture}[scale=.4,yscale=-1,decoration={
markings,
mark=at position .5 with {\arrow[black,line width=.3mm]{stealth}}}]
\draw[help lines] (0,0) grid (14,18);
\draw (2,-1) -- (6,-1); \draw (8,-1) -- (12,-1);
\draw[double,->] (6,-1) -- (6.8,-1); \draw[double,->] (8,-1) -- (7.2,-1);
\filldraw[fill=white] (7,-1) circle (2mm) node[above=1mm] {$\scriptstyle 4$};
\foreach \x in {1,2,3} {
\filldraw[fill=white] (2*\x,-1) circle (2mm) node[above=1mm] {$\scriptstyle\x$}; 
\filldraw[fill=white] (2*7-2*\x,-1) circle (2mm) node[above=1mm] {$\scriptstyle\x$}; }
\foreach \y in {0,2,4,6,8,10,12,14,16,18} {\node at (-1,\y) {$\scriptstyle\y$};}
\begin{scope}[every node/.style={minimum size=.1cm,inner sep=0mm,fill,circle}]
\draw[thick]  (0,10)node {}  -- (2,8)node {}  -- (4,6)node {} -- (6, 8)node{} -- (7, 9.2)node{};
 \draw[thick] (7,8.8)node{} -- (8, 10)node{} -- (10, 12)node{} -- (12, 10)node{} -- (14,8)node{};
\end{scope}
\end{tikzpicture}
\end{equation*}
\label{figoverlap}
\end{figure}
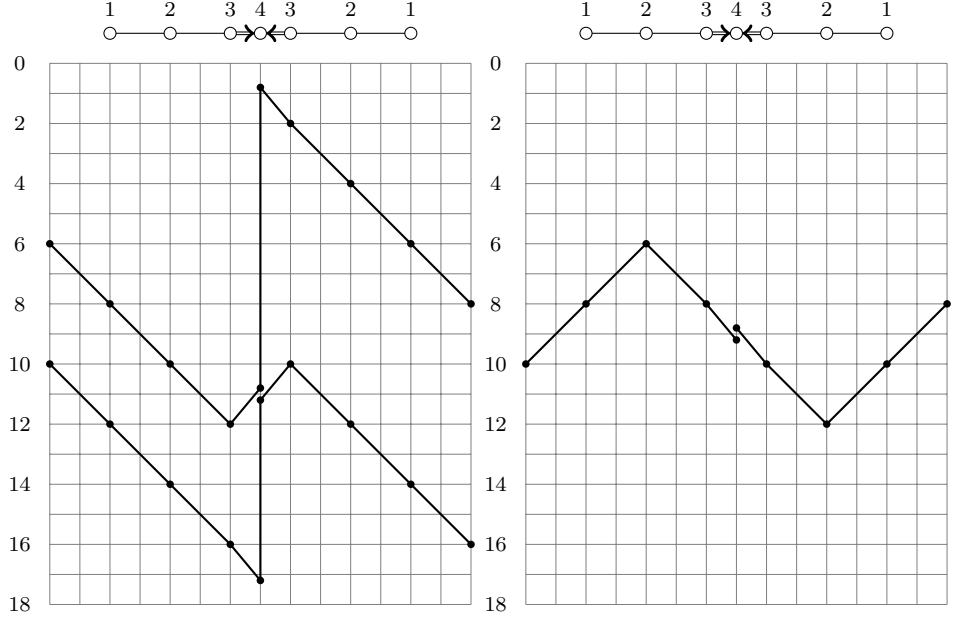

It follows that for any $\ps \in \nops$ and $(j,l)\in \cal W$, at most one of the paths can be lowered at $(j,l)$ and at most one path can be raised at $(j,l)$. Therefore, there is no ambiguity in performing a raising or a lowering move at $(j,l)$ on a non-overlapping tuple of paths $\ps\in\nops$, to yield a new tuple $\ps\scr A_{j,l}^{\pm 1}$.

The following lemma easily follows from the definitions. 
\begin{lemma}
Let $(i_t,k_t)\in \cal X$, $1\leq t\leq T$, be an extended snake of length $T\in \Z_{\geq 1}$ and $\ps\in \nops$. Then $\prod_{t=1}^T\m(p_t)$ is dominant if and only if $p_t$ is the highest path of $\scr P_{i_t,k_t}$ for all $1\leq t \leq T$. Similarly, $\prod_{t=1}^T\m(p_t)$ is anti-dominant if and only if $p_t$ is the lowest path of $\scr P_{i_t,k_t}$ for all $1\leq t \leq T$.\qed
\label{dompath}
\end{lemma}

The following lemma gives information about how overlaps arise when performing a lowering move on a tuple of non-overlapping paths.
\begin{lemma}\label{lowering lemma}
Let $(i_t,k_t)_{1\leq t\leq T}\subseteq \cal X$ be an extended snake of length $T\in \Z_{\geq 1}$ and $\ps\in \nops$. Let $1\leq t \leq T$ and $(j,l)\in \cal X$ be such that the path $p_t$ can be lowered at $(j,l)$. This move introduces an overlap if and only if there is an $s, \ t<s\leq T$, such that $p_s$ has an upper corner at $(j, l+r_j)$ or a lower corner at $(j,l-r_j)$.
\label{overlap condition}
\end{lemma}
\begin{proof}
This is seen by inspection of the definitions above of paths and moves. We illustrate the distinct cases, up to symmetry. In the most cases the overlap occurs at an upper corner $(j, l+r_j)$ of $p_s$:
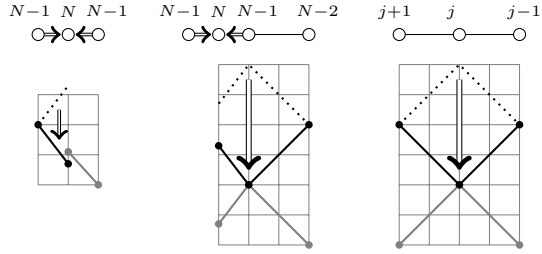
\begin{equation*}
\begin{tikzpicture}[baseline=0cm,scale=.4,yscale=-1]
\begin{scope}[xshift=12cm]
\draw[help lines] (-1,0) grid (3,6);
\draw (-1,-1) -- (3,-1);
\filldraw[fill=white] (-1,-1) circle (2mm) node[above=3mm,left=-3mm] {$\scriptscriptstyle j+1$};
\filldraw[fill=white] (1,-1) circle (2mm) node[above=3mm,right=-3mm] {$\scriptscriptstyle j$};
\filldraw[fill=white] (3,-1) circle (2mm) node[above=3mm,right=-3mm] {$\scriptscriptstyle j-1$};
\draw[thick,dotted]  (-1,2) -- (1,0) -- (3,2) ;
\begin{scope}[every node/.style={minimum size=.1cm,inner sep=0mm,fill,circle}]
\draw[thick,gray] (-1,6)  node {} -- (1,4) node {}  -- (3,6)node {} ;
\draw[thick]  (-1,2) node {} -- (1,4) node {} -- (3,2) node {}  ;
\end{scope}

\draw[double distance=.5mm,->,shorten <= 2mm,shorten >= 2mm] (1,0) -- (1,4);

\end{scope}
\begin{scope}[xshift=5cm]
\draw[help lines] (0,0) grid (3,6);
\draw[double distance =.2mm,->,shorten >= 1mm] (-1,-1) -- (0,-1);
\draw[double distance =.2mm,->,shorten >= 1mm] (1,-1) -- (0,-1);
\draw (1,-1) -- (3,-1);
\filldraw[fill=white] (-1,-1) circle (2mm) node[above=3mm,left=-3mm] {$\scriptscriptstyle N-1$};
\filldraw[fill=white] (0,-1) circle (2mm) node[above=1mm] {$\scriptscriptstyle N$};
\filldraw[fill=white] (1,-1) circle (2mm) node[above=3mm,right=-3mm] {$\scriptscriptstyle N-1$};
\filldraw[fill=white] (3,-1) circle (2mm) node[above=3mm,right=-3mm] {$\scriptscriptstyle N-2$};
\draw[thick,dotted]  (0,1.3) -- (1,0) -- (3,2) ;
\begin{scope}[every node/.style={minimum size=.1cm,inner sep=0mm,fill,circle}]
\draw[thick,gray] (0,5.3) node {} -- (1,4) node {}   -- (3,6)node {} ;
\draw[thick]  (0,2.7)node {}  -- (1,4)node {}  -- (3,2)node {}  ;
\end{scope}
\draw[double distance=.5mm,->,shorten <= 2mm,shorten >= 2mm] (1,0) -- (1,4);
\end{scope}
\begin{scope}[xshift=0cm]
\draw[help lines] (-1,1) grid (1,4);
\draw[double distance =.2mm,->,shorten >= 1mm] (-1,-1) -- (0,-1);
\draw[double distance =.2mm,->,shorten >= 1mm] (1,-1) -- (0,-1);
\filldraw[fill=white] (-1,-1) circle (2mm) node[above=3mm,left=-3mm] {$\scriptscriptstyle N-1$};
\filldraw[fill=white] (0,-1) circle (2mm) node[above=1mm] {$\scriptscriptstyle N$};
\filldraw[fill=white] (1,-1) circle (2mm) node[above=3mm,right=-3mm] {$\scriptscriptstyle N-1$};
\draw[thick,dotted]  (-1,2) -- (0,.7) ;
\begin{scope}[every node/.style={minimum size=.1cm,inner sep=0mm,fill,circle}]
\draw[thick,gray]  (0,2.9)node {}  -- (1,4)node {}  ;
\draw[thick]  (-1,2) node {} -- (0,3.3)node {}  ;
\end{scope}
\draw[double,->,shorten <= 2mm,shorten >= 2mm] (-.3,1) -- (-.3,3);
\end{scope}
\end{tikzpicture}
\end{equation*}

The exception is when the upper corner $(j, l-r_j)$ of $p_t$ is also a lower corner of $p_s$, which can happen only when $j=N$, cf.  Lemma \ref{overlapcond}(ii): 
\begin{equation*}\begin{tikzpicture}[baseline=0cm,scale=.4,yscale=-1]
\begin{scope}[xshift=0cm]
\draw[help lines] (-1,1) grid (1,4);
\draw[double distance =.2mm,->,shorten >= 1mm] (-1,-1) -- (0,-1);
\draw[double distance =.2mm,->,shorten >= 1mm] (1,-1) -- (0,-1);
\filldraw[fill=white] (-1,-1) circle (2mm) node[above=3mm,left=-3mm] {$\scriptscriptstyle N-1$};
\filldraw[fill=white] (0,-1) circle (2mm) node[above=1mm] {$\scriptscriptstyle N$};
\filldraw[fill=white] (1,-1) circle (2mm) node[above=3mm,right=-3mm] {$\scriptscriptstyle N-1$};
\draw[thick,dotted]  (-1,3) -- (0,1.7) ;

\draw[double,->,shorten <= 2mm,shorten >= 2mm] (-.3,2) -- (-.3,4);

\begin{scope}[every node/.style={minimum size=.1cm,inner sep=0mm,fill,circle}]
\draw[thick,gray]  (0,2.3) node {} -- (1,1) node {} ;
\draw[thick]  (-1,3) node {} -- (0,4.3) node {} ;
\end{scope}

\end{scope}
\end{tikzpicture}
\end{equation*}

The only other possible overlapping scenario is as shown. 
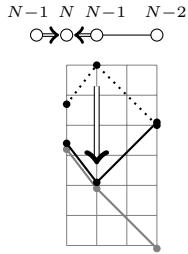
\begin{equation*}\begin{tikzpicture}[baseline=0cm,scale=.4,yscale=-1,xscale=-1]
\draw[help lines] (-3,-2) grid (0,4);
\draw[double distance =.2mm,->,shorten >= 1mm] (-1,-3) -- (0,-3);
\draw[double distance =.2mm,->,shorten >= 1mm] (1,-3) -- (0,-3);
\draw (-3,-3) -- (-1,-3);
\filldraw[fill=white] (1,-3) circle (2mm) node[above=3mm,left=-3mm] {$\scriptscriptstyle N-1$};
\filldraw[fill=white] (0,-3) circle (2mm) node[above=1mm] {$\scriptscriptstyle N$};
\filldraw[fill=white] (-1,-3) circle (2mm) node[above=3mm,right=-3mm] {$\scriptscriptstyle N-1$};
\filldraw[fill=white] (-3,-3) circle (2mm) node[above=3mm,right=-3mm] {$\scriptscriptstyle N-2$};
\begin{scope}[every node/.style={minimum size=.1cm,inner sep=0mm,fill,circle}]
\draw[thick,gray]  (-3,4.1)node {}  -- (-1,2.1) node {} -- (0,.8)node {} ;
\draw [thick, dotted] (-3,0)node {}  -- (-1,-2)node {}  -- (0,-.7)node {} ;
\draw[thick]  (-3,-.1)node {}  -- (-1,1.9)node {}  -- (0,.6)node {}  ;
\draw[double distance=.5mm,->,shorten <= 2mm,shorten >= 2mm] (-1,-1.8) -- (-1,1.8);
\end{scope}
\end{tikzpicture}
\end{equation*}
\label{illegalmove}
This situation occurs if and only if $(N-1,k)$ is an upper corner of some path $p$, $(N-1, k+4)\in p'$ for some $p'$ and $(N-1,k+4)$ is not an upper corner of $p'$. 
However, we claim this does not happen for extended snakes.

Indeed, let $(i,k)\in \cal X$ and $(i',k')\in \cal X$ such that $k'\geq k$. Let $p\in\mathscr P_{i,k}, p'\in \mathscr P_{i'k'}$. 

If the extended snake also has a point $(i'',k'')$ with $k'>k''>k$ then the overlap above is impossible.
Otherwise we use use the definition of the extended snake modules. In the case of (\ref{parity1}), we have:
\begin{equation}\label{e:extsnakemot1}
k'-k \geq 2+ 2i' + 2i -\delta_{iN} - \delta_{i'N}.
\end{equation}
It follows that if for some $\ell\in \Z$, $(N-1,\ell)\in C_{p,+}$ for some $p\in \scr P_{i,k}$, then $(N, \ell+3)\notin p'$, for all $p'\in \scr P_{i',k'}$. 

The case of (\ref{parity0}), 
is similar with the use of equation
\begin{equation}\label{e:extsnakemot}
k'-k \geq 4 + 2i' - 2i -\delta_{iN} - \delta_{i'N}.
\end{equation}
\end{proof} 

The property of the tuples of paths described by this lemma is used for the proof of Theorem \ref{mainthm1}. Informally, it means that the first overlap between paths always corresponds, in the $q$-character, to an illegal lowering step in some $U_q(\hlie {sl}_2)$ evaluation module. 

\begin{rem}\label{extended def rem} The definition of the extended snakes is exactly the combination of
(\ref{e:domnonoverlap1}), (\ref{e:extsnakemot1}) in the case of (\ref{parity1}) and 
of (\ref{e:domnonoverlap}) in the case of (\ref{parity0}). Observe that \eqref{e:domnonoverlap} implies (\ref{e:extsnakemot}) in the latter case.
\end{rem}

\subsection{The $q$-character of an extended snake module} 
Many properties of snakes can be generalized to extended snakes. After Lemma \ref{lowering lemma} is established, 
the proofs are the same as in \cite{MY12a}, Section 5.6. We list a few such properties. 

\begin{lemma}
Let $(i_t,k_t) \in \cal X$, $1\leq t \leq T$, $T\in\Z_{\geq 1}$, be an extended snake. The map
$$\nops \rightarrow \Z[Y_{i,k}^{\pm 1}]_{(i,k)\in \cal X}, \quad \ps \mapsto \prod_{t=1}^{T}\m(p_t),$$
is injective.\qed
\label{monpathcor}
\end{lemma}

\begin{lemma}
Let $(i_t,k_t)\in \cal X$, $1\leq t\leq T$, $T\in \Z_{\geq 1}$, be an extended snake and $\ps \in \nops$. Let $m = \prod_{t=1}^{T}\m(p_t)$. If $mA_{i,k}^{-1}$ is not of the form $\prod_{t=1}^T \m(p_t')$ for any $(p_1', \ldots, p_T')\in \nops$, then neither $mA_{i,k}^{-1}A_{j,l}$ unless $(j,l) = (i,k)$. \qed
\label{excl}
\end{lemma}

\begin{lemma}
Let $(i_t,k_t)\in \cal X$, $1\leq t \leq T$, $T\in \Z_{\geq 1}$, be an extended snake and $\ps\in \nops$. Pick and fix an $i\in I$. Let $\ol{\scr P}\subseteq \nops$ be the set of those non-overlapping tuples of paths that can be obtained from $\ps$ by performing a sequence of raising or lowering moves at points of the form $(i,l)\in \cal W$. Then $\sum_{\ps\in \ol{\scr P}}\beta_i(\prod_{t=1}^T \m(p_t))$ is the $q$-character of simple finite-dimensional $U_q(\hlie{sl}_2)$-module.\qed
\label{sl2inside}
\end{lemma}

\medskip

We are now prepared to prove the first main result of this paper.
\begin{thm}
Let $(i_t,k_t)\in \cal X$, $1\leq t\leq T$, be an extended snake of length $T\in\Z_{\geq 1}$ and $m_+= \prod_{t=1}^T Y_{i_t,k_t}$. Then
\begin{equation}\label{e:mainthm1}\chi_q(L(m_+)) = \sum_{\ps\in \nops} \prod_{t=1}^T\m(p_t).
\end{equation}
In particular $L(m_+)$ is thin, tame, special and anti-special.
\label{mainthm1}
\end{thm}
\begin{proof}
By Lemma \ref{dompath} and the definition of highest path we have that $m_+ = \prod_{t=1}^T\m(p_{i_t,k_t}^+)$. 
By Lemma \ref{monpathcor}, monomials $\prod_{t=1}^T\m(p_t)$ with $\ps\in\nops$ are all distinct. Set
$$\cal M := \left\{\prod_{t=1}^T\m(p_t)| \ps\in\nops\right\}.$$

We show that the conditions of Theorem \ref{thincriteria} apply to the pair $(m_+,\cal M)$. 
In fact, property (i) follows from Lemma \ref{dompath}. Property (ii) is Lemma \ref{excl} and property (iii) is Lemma \ref{sl2inside}. Equation \eqref{e:mainthm1} and the properties of being thin and special follow from Theorem \ref{thincriteria}. Anti-special property follows from Lemma \ref{dompath}.
\end{proof}

As a byproduct we have a description of prime extended snake modules.

\begin{cor}\label{prime}
Let $(i_t,k_t)\in \cal X$, $1\leq t\leq T$, be an extended snake of length $T\in\Z_{\geq 1}$ and $m_+:= \prod_{t=1}^T Y_{i_t,k_t}$. Let $1\leq s<T$ and let $m_1:= \prod_{t=1}^s Y_{i_t,k_t}$, $m_2:= \prod_{t=s+1}^T Y_{i_t,k_t}$. 

Then $L(m_+)=L(m_1)\otimes L(m_2)$ if and only if 
\begin{equation}\label{tensor}
k_{s+1}-k_s\geq 4+2i_s+2i_{s+1} -\delta_{Ni_s} -\delta_{Ni_{s+1}},
\end{equation}
with $k_{s+1}-k_s \equiv 2(i_s-i_{s+1}) -\delta_{Ni_s} -\delta_{Ni_{s+1}}\mod 4$, or
\begin{equation}\label{tensor1}
k_{s+1}-k_s\geq 4N +2 - 2|i_s-i_{s+1}| -\delta_{Ni_s} -\delta_{Ni_{s+1}}, 
\end{equation}
with $k_{s+1}-k_s \equiv 2+2(i_s-i_{s+1}) -\delta_{Ni_s} -\delta_{Ni_{s+1}} \mod 4$.

In particular, the extended snake module is prime if and only if (\ref{tensor}) and \eqref{tensor1} fail for all $s=1,\dots,T-1$.
\end{cor}
\begin{proof}
The inequalities (\ref{tensor}) and \eqref{tensor1} are equivalent to the requirement that
any path  $p\in\mathscr P_{i_s,k_s}$ is strictly above any path $p'\in\mathscr P_{i_{s+1},k_{s+1}}$, in their respective parity cases. The corollary follows from
Theorem \ref{mainthm1}.
\end{proof}

\section{Tame representations}\label{s:tame}

In this section we classify the simple tame finite-dimensional $\uqh g$-modules when $\g$ is of type $B_N$.

\subsection{The case of $N=2$.} 

\begin{prop}
Let $\g$ be of type $B_2$ and let $m_+\in\mathcal P_{\cal X}^+$. The module $L(m_+)$ is tame if and only if $\cal X(m_+)$ is an extended snake.  
\label{B2 case}
\end{prop}
\begin{proof}
The \emph{if} part follows from Theorem \ref{mainthm1}. We prove the \emph{only if} part. Let $V= L(m_+)$ be tame. 
Let $\cal X(m_+) =(i_t,k_t)_{1\leq t\leq T}$, where $T\in \Z_{\geq 1}$, $(i_t,k_t)\in \cal X$, $1\leq t \leq T$.

By Lemma \ref{r:tameresfactor}, the $U_q(\hlie g_j)$-module $L(\beta_j(m_+))$ is tame, $j=1,2$. Note that $U_q(\hlie g_j)$ is isomorphic to $U_{q_j}(\hlie{sl}_2^{(j)})$. Therefore, by Theorem \ref{NT},
$k_{t+1}\neq k_{t}$ whenever $i_t=i_{t+1}$, $1\leq t\leq T-1$. Hence, it remains to discard the cases
\begin{enumerate}
	\item $i_t = i_{t+1} = 1$ and $k_{t+1} - k_t =2$,
	\item $i_t = i_{t+1} = 2$ and $k_{t+1} - k_t =4$,
	\item $i_{t} \neq i_{t+1}$ and $k_{t+1}-k_{t} \in \{1,3\}$,
\end{enumerate}
where $1\leq t\leq T-1$.

Suppose, by contradiction, $t$ is maximal such that one of the above conditions holds. 
In each case, using Lemma \ref{sl2 char} and \eqref{e:qchev}, we find a monomial $m\in \chi_q(V)$ such that $m$ is $j$-dominant and $u_{j,l}(m)\geq 2$ for some $(j,l)\in \cal X$. By Corollary \ref{c:tameslN}, $L(\beta_j(m))$ is a subfactor of $\res_j V$. By Theorem \ref{NT}, it is not tame, in contradiction with Lemma \ref{r:tameresfactor}.

Suppose that (i) holds. 
Let $r\in \Z_{\geq 0}$ maximal such that $u_{1,k_{t+1}+4j}(m_+)=1$, for all $0\leq j \leq r$. Then 
	$$m =m_+A_{1,k_t+2}^{-1}\left(\prod_{j=0}^{r}A_{1,k_{t+1}+4(r-j)+2}^{-1}\right) \in \chi_q(V).$$
	
One easily checks, cf. \eqref{e:lroots}, that $m$ is $2$-dominant and $u_{2,k_t+3}(m)\geq 2$.
	
Suppose that (ii) holds. 
Then the monomial
	$$m = m_+A_{2,k_t+1}^{-1}A_{1,k_t+3}^{-1}\in \chi_q(V),$$
is $2$-dominant and 
	$u_{2,k_t+4}(m)\geq 2$.
	
Suppose that (iii) holds. 
Set
	$$m = m_+A_{i_t,k_t+r_{i_t}}^{-1} \in \chi_q(V).$$
	Explicitly,
	$$m = \left\{\begin{array}{ll}
	m_+ Y_{1,k_t}^{-1}Y_{1,k_t+4}^{-1}Y_{2,k_t+1}Y_{2,k_t+3} &{\rm if} \quad i_t=1,\\
	m_+ Y_{2,k_t}^{-1}Y_{2,k_t+2}^{-1}Y_{1,k_t+1} &{\rm if} \quad i_t=2.
	\end{array}	\right.$$
If $i_t =1$ then $m$ is $2$-dominant and either $u_{2,k_t+1}(m)\geq 2$ or $u_{2,k_t+3}(m)\geq 2$. 
If $i_t=2$ and $k_{t+1}=k_t+1$
then $m$ is $1$-dominant and $u_{1,k_t+1}(m)\geq 2$. Therefore, it remains to consider the case when $i_t=2$ and $k_{t+1} = k_t+3$. 
Let $r\in \Z_{\geq 0}$ be such that $u_{1,k_{t+1}+4j}(m)>0$, for all $0\leq j \leq r$ and $u_{1,k_{t+1}+4r+4}(m)=0$. Then
	$$m' =m A_{1,k_t+3}^{-1}\left(\prod_{j=0}^{r}A_{1,k_{t+1}+4(r-j)+2}^{-1}\right) \in \chi_q(V),$$
$m'$ is $2$-dominant and $u_{2,k_{t+1}+1}(m')\geq 2$.
\end{proof}

\subsection{Main theorem}

\begin{thm}
Let $\lie g$ be of type $B_N$ and let $m_+\in\mathcal P_{\cal X}^+$. The module $L(m_+)$ is tame if and only if $\cal X(m_+)$ is an extended snake. In particular, all tame irreducible modules are thin, special and anti-special. 
\label{mainthm2}
\end{thm} 
Since by Theorem \ref{mainthm2} all irreducible tame representation are special, the following is immediate.
\begin{cor}\label{c:tamebN}
Let $V$ be a tame representation. For each dominant monomial $m\in \cal M(V)$, $L(m)$ is a tame subfactor of $V$. \qed
\end{cor}
\begin{proof} We prove Theorem \ref{mainthm2} by induction on $N$. 

The case of $N=2$ is given by Proposition \ref{B2 case}. 

Let $N\geq 3$. Let $V= L(m_+)$ and let
$\cal X(m_+) = (i_t,k_t)_{1\leq t \leq T}, \ T\in \Z_{\geq 1}$.

Let $J_1=\{2,3,\ldots, N\}$ and $J_2=\{1,2,\ldots, N-1\}$ be subsets of $I$. Consider the subalgebras $U_q(\hlie g_{J_1})$ and $U_q(\hlie g_{J_2})$ of $\uqh g$, these subalgebras are isomorphic to quantum affine algebras of type $B_{N-1}$ and $A_{N-1}$, respectively. 

By Lemma \ref{r:tameresfactor}, $\L(\beta_{J_1}(m_+))$ is a subfactor of $\res_{J_1} V$, and by induction hypothesis,
\begin{equation}
k_{t+1}-k_{t}\geq 4+2|i_{t+1}-i_{t}| -\delta_{Ni_{t+1}} - \delta_{Ni_{t}} \quad {\rm if}\quad k_{t+1} - k_t \equiv 2|i_{t+1}-i_{t}| -\delta_{Ni_{t+1}} - \delta_{Ni_{t}}\mod 4,
\label{case even restriction}
\end{equation}
or
\begin{equation}
k_{t+1} -k_t\geq 2N + 2|N-i_{t+1}-i_{t}+1|- \delta_{Ni_{t+1}} - \delta_{Ni_{t}} 
\label{case odd restriction}
\end{equation}
for all $t$, $1\leq t \leq T-1$, such that $1\notin\{i_t,i_{t+1}\}$.

Similarly, $L(\beta_{J_2}(m_+))$ is a subfactor of $\res_{J_2} V$, and, hence, by Theorem \ref{NT}, (\ref{case even restriction}) holds 
for all $t$, $1\leq t \leq T-1$, such that $N\notin\{i_t,i_{t+1}\}$.

Suppose, by contradiction, that  $(i_{t+1}, k_{t+1})$ is not in extended snake position to $(i_t,k_t)$ for some $t$, $1\leq t \leq T-1$.
Without loss of generality we can assume that $(i_{s+1}, k_{s+1})$ is in extended snake position to $(i_s,k_s)$ for $t<s\leq T-1$.
For convenience, denote $$(i,k) = (i_t,k_t)\quad {\rm and} \quad (i',k') = (i_{t+1}, k_{t+1}).$$ We divide the argument in the following two cases. 
\begin{enumerate}
\item $k' - k \equiv 2|i'-i| -\delta_{Ni'} - \delta_{Ni}\mod 4$,
\item $k' - k \equiv 2+ 2|i'-i| -\delta_{Ni'} - \delta_{Ni}\mod 4$.
\end{enumerate}

In each case, using Lemma \ref{sl2 char} and \eqref{e:qchev}, we find a $J_{\ell}$-dominant monomial $m\in \chi_q(V)$, for some $\ell =1,2$, and consecutive points $(j,l), (j',l')\in \cal X(\beta_{J_\ell}(m))$, $l'\geq l$, such that one of the following holds:
\begin{itemize}
\item $\ell =1$ and $(j',l')$ is not in extended snake position to $(j,l)$ with respect to the algebra $U_q(\hlie g_{J_1})$,
\item $\ell = 2$ and $(j',l')$ is not in snake position to $(j,l)$ with respect to the algebra $U_q(\hlie g_{J_2})$.
\end{itemize}
By induction hypothesis and Corollary \ref{c:tamebN}, if $\ell=1$, and by Theorem \ref{NT} and Corollary \ref{c:tameslN}, if $\ell=2$, $L(\beta_{J_{\ell}})$ is a subfactor of $\res_{J_{\ell}} V$ which is not tame, thus a contradiction with Lemma \ref{r:tameresfactor}. 

Let (i) hold. By \eqref{case even restriction}, it remains to consider the case $\{i,i'\}=\{1,N\}$. Our assumption implies $0\leq k' - k < 2 +2|N - 1|.$
By the definition of $\cal X$ it is equivalent to 
$$1\leq k'-k\leq 2N -3.$$
Therefore, 
$$m=m_+A_{i,k+r_i}^{-1} \in \chi_q(V).$$
If $i=1$, then $m$ is $J_1$ dominant and let $(j,l) = (2,k+2)$ and $(j',l') = (N,k')$. 
If $i=N$, then $m$ is $J_2$-dominant and let $(j,l)= (N-1,k+1)$ and $(j',k') = (1,k')$. In each case we have that 
$u_{j,l}(m)=1$, $u_{j',l'}(m)=1$ and
\begin{equation*}
-1\leq l'-l \leq 2|j'-j|.
\end{equation*}
This finishes the case (i).

Let (ii) hold. Our assumption implies that $$0\leq k'-k < 2N + 2 + 2|N-i-i'| - \delta_{Ni} - \delta_{Ni'}.$$ By the definition of $\cal X$ it is equivalent to
\begin{equation}
0\leq k'-k \leq 2N -2 + 2|N-i'-i|-\delta_{Ni'} - \delta_{Ni}.
\label{assumption eq}
\end{equation} We split the argument further in the following subcases:
\begin{description}
	\item[a] $i'\neq 1$ and $i=1$,
	\item[b] $i'\neq 1$ and $i\neq 1$,
	\item[c] $i'=1$.
\end{description}

Consider the case $\gbr a$. Set
$$m= m_+A_{1,k+2}^{-1} \in \chi_q(V).$$
Let $(j,l) = (2,k+2)$ and $(j',l') = (i',k')$. Then 
$m$ is $J_1$-dominant, $u_{j,l}(m)=1$, $u_{j',l'}(m)=1$. Moreover, \eqref{assumption eq} implies that 
$$-2\leq l'-l \leq 2N-4 +2|N -j' -j +1| - \delta_{Nj'},$$
completing the proof in this subcase.

Consider the case $\gbr b$. By (\ref{case odd restriction}) and (\ref{assumption eq}) it follows that
\begin{equation}
2N + 2|N-i'-i+1|- \delta_{Ni'} - \delta_{Ni}\leq k'-k \leq 2N -2 +2|N-i'-i| -\delta_{Ni'}- \delta_{Ni}.
\label{eqcase1}
\end{equation}
If $N-i'-i\geq 0$, there is no $k$ and $k'$ satisfying (\ref{eqcase1}).
On the other hand, if $N-i'-i< 0$, then (\ref{eqcase1}) is equivalent to 
\begin{equation}
k'-k = 2i' +2i -2 -\delta_{Ni'}- \delta_{Ni}.
\label{eqcase1a}
\end{equation} 
In particular, \eqref{eqcase1a} implies 
\begin{equation*}
m= m_+ \prod_{j=0}^{i-1}A_{i-j,k+r_i +2j}^{-1}\in \chi_q(V),
\end{equation*}
Explicitly, 
\begin{equation*}
m=\left\{\begin{array}{ll}
m_+Y_{i,k}^{-1}Y_{i+1,k+2}Y_{1, k+2+2i}^{-1} & {\rm if}\quad i\leq N-2,\\
m_+Y_{N-1,k}^{-1}Y_{N,k+1}Y_{N,k+3}Y_{1, k+2+2i}^{-1} & {\rm if}\quad i=N-1,\\
m_+Y_{N,k}^{-1}Y_{N,k+4}Y_{1, k+1+2i}^{-1} & {\rm if}\quad i= N.
\end{array}
\right.
\end{equation*}
Let
$$(j,l)= \begin{cases} (i+1,k+2+\delta_{N,i+1}) & {\rm if}\ i<N,\\ (N, k+4) &{\rm if}\  i=N,
\end{cases}\quad {\rm and}\quad (j',l') = (i',k').$$
One readily checks that $m$ is $J_1$-dominant, $u_{j,l}(m)=1$, $u_{j',l'}(m)=1$ and, by \eqref{eqcase1a}, it follows that
$$l'-l =2N-4 + 2|N-j'-j+1| - \delta_{Nj'} - \delta_{Nj},$$
completing the proof in this subcase.

Consider the case $\gbr c$.
Let $r\in \Z_{\geq 0}$ be such that $(1,k'+4j)\in \cal X(m_+)$ for all $0\leq j\leq r$ and $(1,k'+4r+4)\not\in \cal X(m_+)$. Then
\begin{equation*}
m'=\begin{cases}
m_+\prod_{j=0}^{r}A_{1,k'+4(r-j)+2}^{-1}\in \chi_q(V) & {\rm if}\ i> 1,\\
m_+A_{1,k+2}^{-1}\prod_{j=0}^{r}A_{1,k'+4(r-j)+2}^{-1}\in \chi_q(V) & {\rm if}\ i=1.
\end{cases}
\end{equation*}
Let $$(j,l) = \begin{cases}
(i,k)& {\rm if}\ i>1,\\
(2,k+2) & {\rm if}\ i=1,\\
\end{cases}\quad {\rm and} \quad (j',l') = (2,k'+2).$$

One easily checks that $m$ is $J_1$-dominant, $u_{j,l}(m')=1$ and $u_{j',l'}(m')=1$. 
By induction hypothesis, it follows that 
$$l'-l \geq 2N + 2|N-j'-j+1| -\delta_{Nj}.$$
Equivalently,
\begin{equation}
\begin{cases}
k'-k  \geq 2N -2+ 2|N-i-2+1| -\delta_{Ni}& {\rm if}\ i>1,\\
k'-k \geq 4N -6 & {\rm if}\ i=1.
\end{cases}
\label{eqcase2}
\end{equation}
Relations (\ref{assumption eq}) and (\ref{eqcase2}) together are equivalent to 
\begin{equation}\label{e:i<N rest}k'-k= 2N -2 + 2|N-i-1| - \delta_{Ni}.
\end{equation}

If $i<N$ then \eqref{e:i<N rest} implies that 
$$m= m_+ \left(\prod_{j=0}^{N-i}A_{i+j, k + 2 +2j}^{-1}\right)A_{N,k+2(N-1-i)}^{-1}\in \chi_q(V).$$
Explicitly, $$m=m_+Y_{i,k}^{-1}Y_{i-1, k+2}Y_{N-1,k+2(N-1-i)+2}Y_{N,k+2(N-1-i)+3}^{-1}Y_{N,k+2(N-1-i)+1}^{-1}.$$
Let $(j,l) = (N-1,k+2(N-i))$ and $(j',l') = (1,k')$. Then $m$ is $J_2$-dominant, $u_{j,l}(m)=1$, $u_{j',l'}(m)=1$ and
\begin{eqnarray*}
l'-l =  2|j'-j|.
\end{eqnarray*}

If $i=N$ then \eqref{e:i<N rest} implies that 
$$m = m_+\prod_{j=0}^{N-1}A_{N-j,k+1+2j}^{-1}\prod_{j=0}^{r}A_{1,k'+4(r-j)+2}^{-1} \in \chi_q(V).$$
Let $(j,l) = (N,k+4)$ and $(j',l') = (2,k'+2)$. One easily checks that $m$ is $J_1$-dominant, $u_{j,l}(m)=1$, $u_{j',l'}(m)=1$.
Moreover, \eqref{e:i<N rest} implies that
\begin{eqnarray*}
l'-l = 2N -4 + 2|N-j'-j+1| - 1.
\end{eqnarray*}
This completes the proof in this subcase. The proof of Theorem \ref{mainthm2} is finished.
\end{proof}

\section{Tableaux description of snake modules}\label{s:tableaux}

In this section we define a bijection between super standard skew Young tableaux and paths of some associated snake.  

\subsection{Combinatorial properties of non-overlapping paths}

Define a set $\Alp$ (the \emph{alphabet}) equipped with a total ordering $<$ (\emph{alphabetical ordering}) as follows:
$$\Alp := \{1,2,\dots,N,0,\ol N,\dots, \ol 2,\ol 1 \}, \quad 1<2<\dots<N<0<\ol N<\dots< \ol 2<\ol 1.$$
Given a subset $B\subset\Alp$, assume $B = \{a_1 \leq a_2 \leq \ldots \leq a_m\}$ for some $m\in \Z_{\geq 0}$. For $k\in \Z_{\geq 0}$ define the subsets of $B$
\begin{equation}\label{e:alphrest}
^{[k]}B:= \{a_{k+1}, a_{k+2}, \ldots, a_m\} \quad {\rm and}\quad B^{[k]}:= \{a_1, a_2, \ldots, a_{m-k}\}.
\end{equation}
We understand $^{[k]}B = B^{[k]} = \emptyset$, if $k\geq m$.

Let $(i,k)\in \cal X$ and $p\in \scr P_{i,k}$. 

If $i=N$ and $p =: ((x_r,y_r))_{0\leq r \leq N}$, define 
$$R_p:= \{r | 1\leq r\leq N,\ y_r - y_{r-1}>0\}\subseteq \Alp,\quad {\rm and }\quad \ol R_p := \{\ol r | 1\leq r\leq N,\ y_r - y_{r-1}<0\}\subseteq \Alp.$$ 
If $i<N$, recall that $p$ is given by a pair $(a,\ol a)$ where
$$a\in \scr P_{N, k-(2N-2i-1)}\quad {\rm and}\quad \ol{a} \in \scr P_{N,k+(2N-2i-1)}.$$ 
Let
$$a =: ((x_r, y_{r}))_{0\leq r\leq N}\quad {\rm and} \quad \ol{a} =: ((\ol x_r, \ol y_{r}))_{0\leq r\leq N},$$
and define
$$R_p := \{r | 1\leq r\leq N,\ \ol y_r - \ol y_{r-1}<0\}\subseteq \Alp, \quad {\rm and} \quad \ol{R_p} := \{\ol{r} | 1\leq r\leq N,\ y_r - y_{r-1}>0\} \subseteq \Alp.$$
If $i<N$, we also define
$$S_p := \{r | 1\leq r\leq N,\ y_r - y_{r-1}<0\}\subseteq \Alp, \quad {\rm and} \quad \ol{S_p} := \{\ol{r} | 1\leq r\leq N,\ \ol{y}_r - \ol{y}_{r-1}>0\} \subseteq \Alp.$$
Note that 
$$S_p = \{r | 1\leq r\leq N,\ \ol{r} \notin \ol{R_p}\}\quad {\rm and} \quad \ol{S_p} = \{\ol{r} | 1\leq r\leq N,\ r\notin R_p\}.$$
Clearly $p$ is completely described by the pair of sets $R_p, \ol R_p$, and equally so by $S_p, \ol S_p$, when $i<N$. 

{\bf Example.}
If $p=p_{i,k}^+$, for some $(i,k)\in \cal X$, $i< N$, then $S_p = \{1,2,\ldots, i\}$, $\ol{S_p} = \emptyset$, $R_p= \{1,2,\ldots, N\}$ and $\ol{R_p} = \{\ol{i+1}, \ldots, \ol N\}$. \endit

We denote cardinality of a finite set $A$ by $\#A$.
The next lemma follows from the definition of paths.

\begin{lemma}\label{l:halfpathsN}
Let $(i,k)\in \cal X$ and let $p\in \scr P_{i,k}$. Then
\begin{equation}\label{e:halfpathsN}
\# R_p + \#\ol R_p \geq 2N -i \quad {\rm and}\quad \#S_p + \#\ol S_p \leq i.
\end{equation}\qed
\end{lemma}

Let $(i',k')\in \cal X$ be in snake position to $(i,k)\in \cal X$. We say that $(i',k')$ and $(i,k)$ are {\it shifted by $\sigma\in \Z_{\geq 0}$} if
$$k'-k = 4 + 2|i'-i| + 4\sigma - \delta_{Ni} - \delta_{Ni'}.$$
Observe that $\sigma = 0$ corresponds to the minimal snake position. If $(i_t,k_t)_{1\leq t \leq T}$, $T\in \Z_{\geq 1}$ is a snake, we denote by $\sigma_t$ the shift between $(i_t,k_t)$ and $(i_{t+1},k_{t+1})$.

\begin{lemma}\label{l:halfpaths}
Let $(i',k')\in \cal X$ be in snake position to $(i,k)\in \cal X$ shifted by $\sigma\in \Z_{\geq 0}$.
Let $p\in \scr P_{i,k}, p'\in \scr P_{i',k'}$. 
If $p$ is strictly above $p'$ then
\begin{equation}\label{e:halfpaths}
\#\ol R_p + \# R_{p'}\leq 2N - i + \max\{i-i',0\} + \sigma, 
\end{equation}
\begin{equation}\label{e:olRnops}
\#\ol R_p^{[\sigma + \max\{i'-i, 0\}]} \cap\{\ol 1,\ldots, \ol r\} \leq \#  \ol R_{p'}\cap \{\ol 1 ,\ldots, \ol r\}, \quad r =1,\ldots, N-1,
\end{equation}
\begin{equation}\label{e:Rnops}
\# R_p \cap\{ 1,\ldots, r\} \geq \# ^{[\sigma + \max\{i-i', 0\}]} R_{p'}\cap \{ 1,\ldots, r\}, \quad r =1,\ldots, N-1.
\end{equation}
\end{lemma}
\begin{proof}
Let $y\in \Z \pm \epsilon$ be such that $(2N-1,y) \in p$ and $(2N-1,z)\not\in p$ for $z>y$. Let $y'\in \Z \pm \epsilon$ be such that $(2N-1,y')\in p'$ and $(2N-1,z)\not\in p'$ for $z<y$. 

Let $s,s'$ such that $0\leq s\leq i$, $0\leq s'\leq i'$,
$$4(s-1) \leq y-k-2(N-i+1) + \delta_{Ni} \leq 4s\quad {\rm and } \quad 4(s'-1) \leq y'-k'-2(N-i'+1) + \delta_{Ni'} \leq 4s'.$$ 

Since $p$ and $p'$ do not overlap, we have $y'>y$. It follows that $s-s'\leq \sigma + \max\{i-i', 0\}$.
We also have $\#R_p' = N-s'$ and $\ol R_p = N -i+s$.
Therefore, \eqref{e:halfpaths} follows. 

Suppose that $r$ is the minimal integer such that \eqref{e:olRnops} does not hold. Let $s =\# \ol R_{p'}\cap \{\ol 1, \ldots, \ol r\}$. Then 
$$\# \ol R_{p}\cap \{\ol 1, \ldots, \ol r\} = s+ \sigma + \max\{i'-i,0\}+1.$$
In particular, it follows that 
$$\iota (r,k + 2i -\delta_{Ni} + 4(s+ \sigma + \max\{i-i',0\}+1) -2r) \in p \quad {\rm and} \quad \iota(r, k'+2i' - \delta_{Ni'} + 4s -2r)\in p'.$$
Then paths $p$ and $p'$ overlap at $r$, and hence, a contradiction.

Equation \eqref{e:Rnops} is proved similarly.
\end{proof}

\begin{cor}
\label{r:halfpaths} 
Let $(i',k')\in \cal X$ be in snake position to $(i,k)\in \cal X$ shifted by $\sigma\in\Z_{\geq 0}$. Let $p\in \scr P_{i,k}$, $p'\in \scr P_{i',k'}$ and assume $i,i'< N$. If $p$ and $p'$ do not overlap, then
\begin{equation}\label{e:halfpaths2}
\# S_p + \# \ol S_{p'}\geq  i - \max\{i-i',0\} - \sigma,
\end{equation}
\begin{equation}\label{e:Snops}
\#S_p \cap\{1,\ldots, r\} \geq \# ^{[\sigma + \max\{i'-i, 0\}]} S_{p'}\cap \{1,\ldots, r\}, \quad r =1,\ldots, N-1,\ {\rm and}
\end{equation}
\begin{equation}\label{e:olSnops}
\#\ol S_p^{[\sigma + \max\{i-i', 0\}]} \cap\{\ol 1,\ldots,\ol r\} \leq \# \ol S_{p'}\cap \{\ol 1,\ldots, \ol r\}, \quad r =1,\ldots, N-1.
\end{equation}\qed
\end{cor}

\subsection{Tableaux}
In this paper a \emph{skew diagram} $\sd$ is a finite subset $\sd \subset \Z\times\Z_{> 0}$ such that 
\begin{enumerate}\item if $\sd\neq \emptyset$ then there is a $j\in \Z$ such that $(j,1)\in\sd$, and,
\item if $(i,j)\notin \sd$ then either $\forall i'\geq i, \forall j'\geq j$, $(i',j') \notin \sd$ or $\forall i'\leq i, \forall j'\leq j$, $(i',j') \notin \sd$.
\end{enumerate} 
If $(i,j)\in \sd$ we say $\sd$ has a \emph{box} in \emph{row} $i$, \emph{column} $j$. For each $j\in \Z_{>0}$, let $b_j= \max\{i\in \Z| (i,j)\in \sd\}$, be the \emph{bottom box} in the column $j$ and $t_j = \min\{i\in \Z| (i,j)\in \sd\}$ be the \emph{top box} in the column $j$. Define also the \emph{length} of the column $j$ by $l_j = \#\{i\in \Z| (i,j)\in \sd\}$, and observe that $l_j = b_j - t_j+ 1$. 

A \emph{skew tableau} $\cal T$ with shape $\sd$ is then any map $\cal T:\sd \to \Alp$ that obeys the following horizontal rule (H) and vertical rule (V):
\begin{enumerate} \item[(H)] $\cal T(i,j) \leq \cal T(i,j+1)$ and $(\cal T(i,j),\cal T(i,j+1)) \neq (0,0)$, 
                      \item[(V)] $\cal T(i,j) < \cal T (i+1,j)$ or $(\cal T(i,j),\cal T(i+1,j)) = (0,0).$ \end{enumerate} 

Let $\Tab\sd$ denote the set of tableaux of shape $\sd$.

If a skew diagram contains a rectangle of size $(2N+1)\times 2$ then 
horizontal rule implies that there exists no skew tableau of shape $\sd$. 
We call  a skew diagram {\it super skew diagram} if $\#\{i\in \Z|(i,j)\in \sd,\ (i,j+1)\in \sd\}\leq 2N$ for all $j\in\Z_{>0}$.

From now on we consider only super skew diagrams $\sd$. 

We call a skew diagram {\it generic super skew diagram} if  $\#\{i\in \Z|(i,j)\in \sd,\ (i,j+1)\in \sd\}<2N$ for all $j\in\Z_{>0}$.
In other words, a generic skew diagram contains no rectangles of size $2N\times 2$. 

Let $\sd$ be a super skew diagram.
For each $\cal T\in\Tab\sd$ we associate a monomial in $\Z[Y_{i,k}^{\pm 1}]_{(i,k)\in I\times\Z}$ as follows:
\begin{equation} M(\cal T) := \prod_{(i,j)\in \lambda/\mu} \m(\cal T(i,j),4(j-i)) \label{mtab},
\end{equation}
where the contribution of each box is given by\\

$\begin{array}{rcl}\m: \Alp\times \Z & \to & \Z[Y_{i,k}^{\pm 1}]_{(i,k)\in I\times \Z},\\
(i,k) &\mapsto & Y_{i-1,2i+k}^{-1} Y^{\phantom{+1}}_{i,2i-2+k}  ,  1\leq i\leq N-1,\\ 
(N,k) &\mapsto  & Y_{N-1,2N+k}^{-1}  Y_{N,2N-3+k}Y^{\phantom{+1}}_{N,2N-1+k},\\ 
(0,k) &\mapsto  & Y_{N,2N+1+k}^{-1} Y^{\phantom{+1}}_{N,2N-3+k},\\
(\ol N,k)& \mapsto & Y_{N,2N-1+k}^{-1}Y_{N,2N+1+k}^{-1} Y^{\phantom{+1}}_{N-1,2N-2+k},  \\ 
(\ol i,k)& \mapsto & Y^{-1}_{i,4N-2i+k} Y^{\phantom{+1}}_{i-1,4N-2-2i+k}  , 1\leq i\leq N-1,\\
\end{array}
$

\noindent with convention, $Y_{0,k}:= 1$ and $Y_{N+1,k}:=1$ for all $k\in\Z$. Note that $$\chi_q(\L(Y_{1,0}))=\sum_{i\in \Alp}\m(i,0).$$ 

\subsection{Dominant tableau.} 
Given a super skew diagram $\sd$ we define $\cal T_+: \sd \to \Alp$, by filling up the boxes of $\sd$ with letters in $\Alp$ according to the following rule. Starting from column $1$ and going from the column $j$ to the column $j+1$ in $\sd$, always from the most top empty box and working downwards in the alphabetical order, fill up the column $j$ as follows:
\begin{enumerate}
\item enter letters $1,\ldots, r\leq N$, for the maximum $r$ possible,
\item enter as many successive $0$'s as possible, respecting the horizontal rule for the column $j-1$,
\item enter letters $\ol N, \ldots, \ol r$, for the maximum $\ol r$ possible.
\end{enumerate}
By construction, one easily checks that
\begin{lemma}
The map $\cal T_+$ is a tableau.\qed
\end{lemma}

The tableau $\cal T_+$ is called the \emph{dominant tableau} of shape $\sd$. Let $m_+=m_+\sd := M(\cal T_{+})$ obtained by (\ref{mtab}). We give an alternative way of computing $m_+$. For each column $j$ of $\sd$ such that $l_j\geq N$, let $s_j = t_j + N-1$, and observe that $\cal T_+(s_j,j) = N$. A column $j$ of $\sd$ is said to be \emph{special} if $l_j\geq N$ and $(s_j+1,j+1)\notin \sd$. Define $$\scr S = \scr S\sd : = \{j\ | \textrm{	the column}\ j \ \textrm{of}\ \sd \ \textrm{is special}\}.$$

\begin{lemma}\label{l:dommontbl}
Let $\sd$ be a skew diagram. Then
\begin{equation}
m_+\sd = \prod_{j\in \Z_{>0}} \bt(\cal T_+(b_j, j),4(j-b_j))\prod_{j\in \scr S}\bt(N, 4(j-s_j)+2),
\label{dom mon tbl}
\end{equation}
where $\bt: \Alp\times\Z \to \Z[Y_{i,k}]_{(i,k)\in I \times \Z}$ maps
  
\noindent 
$\begin{array}{rcl}
(i,k) & \mapsto & Y^{\phantom{+1}}_{i,2i-2+k},\,  1\leq i\leq N-1,\\ 
(N,k) & \mapsto & Y_{N,2N-3+k},\\ 
(0,k) & \mapsto & Y^{\phantom{+1}}_{N,2N-3+k},\\
(\ol i,k) & \mapsto & Y^{\phantom{+1}}_{i-1,4N-2-2i+k},\, 1\leq i\leq N.
\end{array}$

In particular, $m_+\sd\in \cal P_{\cal X}$.\qed
\end{lemma}

\begin{figure}\caption{A non-generic super skew diagram and its dominant tableau in type $B_2$.}
\begin{center}
\begin{equation*}\begin{tikzpicture}[scale=.4,yscale=-1]
\begin{scope}[xshift=-.5cm,yshift=-.5cm] \draw[help lines] (1,-6) grid (5,1); \end{scope}
\foreach \y in {-5,-4,-3,-2,-1,0,1} {\node at (-.5,\y-1) {$\scriptstyle\y$};}
\foreach \x in {1,2,3,4,} {\node at (\x,-7) {$\scriptstyle\x$};}
\begin{scope}[every node/.style={minimum size=.4cm,inner sep=0mm,fill=gray!10,draw,rectangle}]
\foreach \x/\y in {1/0,1/-1,1/-2,1/-3,2/0,2/-1,2/-2,2/-3,3/-1,3/-3,3/-2,3/-3,3/-4, 3/-5 ,4/-6, 4/-4, 4/-5} {\node at (\x,\y) {};}
\end{scope}
\end{tikzpicture}\quad\text{and}\quad
\begin{tikzpicture}[scale=.4,yscale=-1]
\begin{scope}[xshift=-.5cm,yshift=-.5cm] \draw[help lines] (1,-6) grid (5,1); \end{scope}
\foreach \y in {-5,-4,-3,-2,-1,0,1} {\node at (-.5,\y-1) {$\scriptstyle\y$};}
\foreach \x in {1,2,3,4,} {\node at (\x,-7) {$\scriptstyle\x$};}
\begin{scope}[every node/.style={minimum size=.4cm,inner sep=0mm,draw,rectangle}]
\foreach \x/\y/\c in {1/0/0,1/-1/0,1/-2/2,1/-3/1,2/0/$\ol 1$,2/-1/$\ol 2$,2/-2/2,2/-3/1,3/-1/$\ol 2$,3/-2/0,3/-3/0,3/-4/2, 3/-5/1 ,4/-6/1, 4/-4/0, 4/-5/2} {\node at (\x,\y) {\c};}
\end{scope}
\end{tikzpicture}\,\,
\end{equation*}
\end{center}
\label{f:ng super skew diagram}
\end{figure}

{\bf Example.} In type $B_2$, the non-generic super skew diagram shown in Figure \ref{f:ng super skew diagram} has $\scr S =\{3,4\}$, the column $2$ is non-generic (cf. Section \ref{non generic sec}), and the dominant monomial associated to $\cal T_{+}$ is 
$$Y_{2,1}Y_{1,14}Y_{2,27}Y_{2,29}Y_{2,35}.$$

Later we show that, for each $\cal T\in \Tab\sd$, $M(\cal T)\in \cal P^+$ iff $\cal T = \cal T_+$, see Theorem \ref{t:tabxpaths}. 

Note that a super skew diagram is non-generic if and only if there exist $j$ such that $\cal T_+(b_j,j)=\bar 1$.
We now focus on generic skew diagrams. The non-generic ones are treated in Section \ref{non generic sec}.

Let $\sd$ be generic super skew diagram and let column $j$ be non-empty. Define
$$\varsigma_j = j +\#\{ k\in \scr S\sd\ |\ k<j\}.$$
Then the column $j$ contributes to $m_+(\sd)$ a variable $Y_{i_{\varsigma_j},k_{\varsigma_j}}$ , and if column $j$ is special then it also contributes a variable $Y_{N,k_{\varsigma_j+1}}$.

The following lemma is a consequence of Lemma \ref{l:dommontbl}.
\begin{lemma}\label{l:possnktab}
Let $\sd$ be a generic super skew diagram and $j,j'$ columns of $\sd$. If $j'>j$ then 
$$k_{\varsigma_{j'}} \geq k_{\varsigma_j+1} > k_{\varsigma_j}.$$
Moreover, the equality holds only if $j' = j+1$ and $j$ is not special.\qed  
\end{lemma}

\begin{lemma}\label{l:domtabsnk}
Let $\sd$ be a generic super skew diagram. Then the sequence $\cal X(m_+\sd)$ is a snake.
\end{lemma}
\begin{proof}
By Lemmas \ref{l:dommontbl}  and \ref{l:possnktab} it suffices to prove that the monomials corresponding to two consecutive columns are in snake position. Let $j$ and $j+1$ two columns of $\sd$. By inspection we prove that: 

If $j$ is not special, then $(i_{\varsigma_{j+1}},k_{\varsigma_{j+1}})$ is in snake position to $(i_{\varsigma_{j+1}},k_{\varsigma_{j+1}})$, and
$$\sigma_{\varsigma_j} = \begin{cases} 
b_j - b_{j+1} -\max\{i_{\varsigma_j} - i_{\varsigma_{j+1}},0\} & {\rm if} \ l_j<N, \\
b_j - b_{j+1} - \max\{i_{\varsigma_{j+1}} - i_{\varsigma_{j}},0\} & {\rm if} \ l_j>N.\\
\end{cases}$$

If $j$ is special, then $(i_{\varsigma_{j+1}}, k_{\varsigma_{j+1}})$ is in snake position to $(N, k_{\varsigma_{j}+1})$, and the latter is in snake position to $(i_{\varsigma_{j}},k_{\varsigma_{j}})$. Moreover, 
$$\sigma_{\varsigma_j} = l_j - 2N + i_{\varsigma_j} \quad {\rm and} \quad \sigma_{\varsigma_j+1} = s_j - b_{j+1} - N +i_{\varsigma_{j+1}}.$$
\end{proof}

\subsection{Bijection between paths and tableaux}
Let $\sd$ be a generic super skew diagram. Let $(i_t,k_t)_{1\leq t \leq T}:= \cal X(m_+\sd)$. By  Lemma \ref{l:dommontbl}, 
$T$ is the number of non-empty columns plus the number of special columns.
Given $\ps \in \nops$, we write simply $R_t$ instead of $R_{p_t}$ and similarly so for $\ol R_t$, $S_t$ and $\ol S_t$. 

For $\ps \in \nops$, we define the map $\cal T_{\ps} : \sd \to \Alp$ by filling up the boxes of $\sd$. Each column $j$ is filled up by the following way:

\begin{enumerate}
\item if $l_j<N$, starting from the box $(t_j,j)$ and working downwards, enter the letters of $S_{\varsigma_j}$ in alphabetical order. Then, starting from the box $(b_j,j)$ and working upwards enter the letters of $\ol S_{\varsigma_j}$ in reverse alphabetical order. Enter the letter $0$ into all the boxes in the $j^{\rm th}$ column that remain unfilled,
\item if $l_j\geq N$, start from the box $(t_j,j)$ and working downwards, enter the letters of $R_{\varsigma_j+1}$ in alphabetical order. Then, starting from the box $(b_j,j)$ and working upwards enter the letters of $\ol R_{\varsigma_j}$ in reverse alphabetical order. Enter the letter $0$ into all the boxes in column $j$ that remain unfilled.
\end{enumerate}

\begin{prop}\label{p:bijpathtab}
The map $\nops\to \Tab\sd$ sending $\ps\mapsto \cal T_{\ps}$ is a bijection and 
\begin{equation}\label{e:equalmon}
\m (\ps) = M (\cal T_{\ps}).
\end{equation}
\end{prop}
\begin{proof}
Let for brevity $\cal T=\cal T_{\ps}$.
We show that $\cal T$ is, in fact, a tableau. First, observe that no box of the column $j$ is filled up twice. If $l_j<N$,  it follows  from \eqref{e:halfpathsN}. If $l_j\geq N$, it follows from
\eqref{e:halfpaths} and from relation
$$l_j = 2N-i_{\varsigma_j} + \max\{i_{\varsigma_j}-i_{\varsigma_j+1},0\} + \sigma_{\varsigma_j}.$$

Moreover, $\cal T$ respects the vertical rule (V), by construction. To prove that $\cal T$ respects the horizontal rules (H), it suffices to study $\cal T(i,j)$ and $\cal T(i,j+1)$ for each $(i,j)\in \sd$. 
We split the analysis in the following cases:
\begin{enumerate}[a]
\item[(a)] $l_j< N$ and $l_{j+1}< N$,
\item[(b)] $l_j>N$ and $j\notin \scr S$,
\item[(c)] $j\in \scr S$ and $l_{j}\geq N$,
\item[(d)] $j\in \scr S$ and $l_{j+1}<N$,
\item[(e)] $l_j< N$ and $l_{j+1}\geq N$.
\end{enumerate}

Let $B_j\subseteq\{1,\ldots, N\}$ and $\ol B_j\subset\{\ol 1,\ldots, \ol N\}$ (resp. $B_{j+1}$ and $\ol B_{j+1}$) be the sets which fill up the column $j$ (resp. $j+1$) by the above procedure. In each case, we prove that

\begin{enumerate}
\item $\# B_j\cap \{1,\ldots, r\} \geq \# ^{[t_j-t_{j+1}]}B_{j+1}\cap \{1,\ldots, r\}$, $r=1,\ldots, N-1$,
\vspace{0.3 cm}
\item $\# \ol B_j^{[b_j-b_{j+1}]}\cap \{\ol 1,\ldots, \ol r\} \leq \# \ol B_{j+1}\cap \{\ol 1,\ldots, \ol r\}$, $r=1,\ldots, N-1$,
\vspace{0.3 cm}
\item $\# B_j + \# \ol B_{j+1}\geq b_{j+1}-t_j +1$.
\end{enumerate}
In particular, (i)-(iii) implies that $j$ and $j+1$ respect the horizontal rules.

Consider the case (a). By Lemma \ref{l:domtabsnk}, we have that $b_{j}-b_{j+1} = \sigma_{\varsigma_{j}} + \max\{i_{\varsigma_{j}} - i_{\varsigma_{j}+1},0\}$. Then $t_{j}-t_{j+1} = \sigma_{i_{\varsigma_j}} + \max\{i_{\varsigma_j +1}- i_{\varsigma_j},0\}$, and $b_{j+1} - t_{j} = i_{\varsigma_{j}}- \max\{i_{\varsigma_{j}} - i_{\varsigma_{j}+1},0\} -\sigma_{\varsigma_j} -1$. 
Therefore (i), (ii) and (iii) follow, respectively, by \eqref{e:Snops}, \eqref{e:olSnops} and \eqref{e:halfpaths2}.

Consider the case (b). By Lemma \ref{l:domtabsnk}, we see that $b_{j}-b_{j+1} = \sigma_{\varsigma_j} + \max\{i_{\varsigma_{j+1}} - i_{\varsigma_j}, 0\}$, 
$t_{j}- t_{j+1} = \sigma_{\varsigma_j+1} + \max\{i_{\varsigma_{j+1}} - i_{\varsigma_{j+2}}, 0\}$, and $b_{j+1}-t_j = 2N-i_{\varsigma_{j+1}} -1$. 
Therefore, (i), (ii) and (iii) follow, respectively, by \eqref{e:Rnops}, \eqref{e:olRnops} and \eqref{e:halfpathsN}.

Consider the case (c). By Lemma \ref{l:halfpaths}, we have 
\begin{equation}\label{e:caseiiit}
\# \ol R_{\varsigma_{j}}^{[\sigma_{\varsigma_j}]}\cap \{\ol 1, \ldots, \ol r\} \leq \# \ol R_{\varsigma_{j}+1}\cap \{\ol 1, \ldots, \ol r\},
\end{equation}
and
\begin{equation}\label{e:caseiiit1}
\# \ol R_{\varsigma_{j}+1}^{[\sigma_{\varsigma_j+1}]}\cap \{\ol 1, \ldots, \ol r\} \leq \# \ol R_{\varsigma_{j+1}}\cap \{\ol 1, \ldots, \ol r\},
\end{equation}
for $r=1,\ldots, N-1$. Therefore, inequalities \eqref{e:caseiiit} and \eqref{e:caseiiit1} combined imply
\begin{equation}\label{e:caseiiit2}
\# \ol R_{\varsigma_{j}}^{[\sigma_{\varsigma_j}+\sigma_{\varsigma_j+1}]}\cap \{\ol 1, \ldots, \ol r\} \leq \# \ol R_{\varsigma_{j+1}}\cap \{\ol 1, \ldots, \ol r\},\quad r=1,\ldots, N-1.
\end{equation}
Moreover, since $l_j = b_j -t_{j}+1$ and $s_j-t_j = N-1$, Lemma \ref{l:domtabsnk} implies that 
$$\sigma_{\varsigma_{j}} + \sigma_{\varsigma_{j}+1} + N - i_{\varsigma_{j}} = b_{j}-b_{j+1},$$
and then, (ii) is proved in this case. We prove (i) similarly, using that
\begin{equation}\label{e:caseiiib}
\# R_{\varsigma_{j}+1}\cap \{1, \ldots, r\} \geq \# ^{[\sigma_{\varsigma_{j}+1} + \sigma_{\varsigma_{j+1}} + N - i_{\varsigma_{j+1}+1}]}R_{\varsigma_{j+1}+1}\cap \{1, \ldots, r\}, \quad r=1,\ldots, N-1,
\end{equation}
and $\sigma_{\varsigma_{j}+1} + \sigma_{\varsigma_{j+1}} + N - i_{\varsigma_{j+1}+1} = t_{j}-t_{j+1}$.

Since $i_{\varsigma_j+1} = i_{\varsigma_{j+1}} = N$, by \eqref{e:halfpaths} it follows that $\# \ol R_{\varsigma_{j}+1} + \# R_{\varsigma_{j+1}} \geq N + \sigma_{\varsigma_{j}+1}$. Therefore (iii) follows by observing that $b_{j+1}-t_j \leq N-1$.

Consider the case (d). By Lemma \ref{l:halfpaths}, we have 
$$\#\ol R_{\varsigma_{j}+1}^{[\sigma_{\varsigma_{j}+1}]}\cap \{\ol 1, \ldots, \ol r\} \leq \#\ol R_{_{\varsigma_{j+1}}}\cap\{\ol 1, \ldots, \ol r\},\quad r=1,\ldots, N-1.$$
However, since $i_{\varsigma_{j}+1} = N$, the above inequality is equivalent to 
\begin{equation}\label{e:caseiv}
\# R_{\varsigma_{j}+1}\cap \{1,\ldots, r\} \geq \#^{[\sigma_{\varsigma_{j}+1}]}S_{\varsigma_{j+1}}\cap\{1,\ldots, r\}, \quad r=1,\ldots, N-1.
\end{equation}
Moreover, by Lemma \ref{l:domtabsnk}, we prove that $t_j-t_{j+1} = \sigma_{\varsigma_{j}+1}$, using that $s_j = t_j + N -1$, $b_{j+1}-i_{\varsigma_{j+1}} = t_{j+1} -1$. Item (i) follows in this case.

By Lemma \ref{l:halfpaths}, we also have 
\begin{equation}\label{e:caseivb}
\# \ol R_{\varsigma_{j}}^{[\sigma_{\varsigma_{j}}+N - i_{\varsigma_{j}}]}\cap \{\ol 1,\ldots, \ol r\} \leq  \# \ol R_{\varsigma_{j}+1}\cap\{\ol 1,\ldots, \ol r\},
\end{equation}
and
\begin{equation}\label{e:caseivb1}
\# R_{\varsigma_{j}+1}\cap \{1,\ldots, r\} \geq \#^{[\sigma_{\varsigma_{j}+1}+N - i_{\varsigma_{j+1}}]}R_{\varsigma_{j+1}}\cap\{1,\ldots, r\},
\end{equation}
for $r=1,\ldots, N-1$.
Since $i_{\varsigma_{j}+1}=N$, \eqref{e:caseivb1} is equivalent to
\begin{equation}\label{e:caseivb2}
\# \ol R_{\varsigma_{j}+1}^{[\sigma_{\varsigma_{j}+1}+N - i_{\varsigma_{j+1}}]}\cap \{\ol 1,\ldots, \ol r\} \leq \# \ol S_{\varsigma_{j+1}}\cap\{\ol 1,\ldots, \ol r\},\quad r=1,\ldots, N-1.
\end{equation}
Inequalities \eqref{e:caseivb} and \eqref{e:caseivb2} combined imply
\begin{equation*}
\# \ol R_{\varsigma_{j}}^{[\sigma_{\varsigma_{j}}+N - i_{\varsigma_{j}}+\sigma_{\varsigma_{j}+1}+N - i_{\varsigma_{j+1}}]}\cap \{\ol 1,\ldots, \ol r\} \leq \#\ol S_{\varsigma_{j+1}}\cap\{\ol 1,\ldots, \ol r\},\quad r=1,\ldots, N-1.
\end{equation*}
Moreover, by Lemma \ref{l:domtabsnk}, we see that
$$b_j-b_{j+1} = \sigma_{\varsigma_{j}}+N - i_{\varsigma_{j}}+\sigma_{\varsigma_{j}+1}+N - i_{\varsigma_{j+1}},$$
using that $l_j = b_j-t_j +1$ and  $s_j= t_j +N -1$, and then, (ii) holds in this case.

By \eqref{e:halfpaths}, we have
$$\# \ol R_{\varsigma_j +1} + \# R_{\varsigma_{j+1}}\leq 2N - i_{\varsigma_j +1} + (N- i_{\varsigma_{j+1}}) + \sigma_{\varsigma_{j}+1},$$
which is equivalent to
$$N -\# R_{\varsigma_j +1} + N -\#\ol S_{\varsigma_{j+1}}\leq 2N - i_{\varsigma_j +1} + (N- i_{\varsigma_{j+1}}) + \sigma_{\varsigma_{j}+1}.$$
Since $i_{\varsigma_j+1} = N$, we have 
$$\# R_{\varsigma_j +1} + \#\ol S_{\varsigma_{j+1}}\geq i_{\varsigma_{j+1}} - \sigma_{\varsigma_{j}+1}.$$
Since $b_{j+1}-t_j +1 = i_{\varsigma_{j+1}}- \sigma_{\varsigma_{j}+1}$, (iii) holds in this case

Case (e) is proved similarly to case (d).

Thus for $\ps\in \nops$, the map $\cal T =\cal T_{\ps}$, as defined above, is a tableau. Since each path $p_t$ is completely described by the pair of sets $R_t$ and $\ol R_t$, and the above description can be made backwards in order to obtain the sequence $\ps_{\cal T}$ from a tableaux $\cal T$, we have a bijection between $\nops$ and $\Tab\sd$.

It remains to show that \eqref{e:equalmon} holds. By Lemma \ref{dompath} and the definition of highest path we have $m_+\sd = \prod_{t=1}^T\m (p_{i_t,k_t}^+)$.  Since each path $\ps$ is obtained from $(p_{i_1,k_1}^+,\ldots, p_{i_T,k_T}^+)$ by applying a sequence of lowering moves $\scr A_{j_s,l_s}^{-1}, 0\leq s \leq S$, for some $S\in \Z_{\geq 0}$, it suffices to check that if $(i,k)\in \cal W$ is such that $\ps \scr A_{i,k}^{\pm 1}\in \nops$, then
$$M (\cal T_{\ps\scr A_{i,k}^{\pm 1}}) = M(\cal T_{\ps})A_{i,k}^{\pm 1}.$$
It is easily done by inspection. 

The proof of the proposition is finished.
 
\end{proof}

\begin{thm}\label{t:tabxpaths}
Let $\sd$ be a generic super skew diagram and $m_+ = m_+\sd$. Then
$$\chi_q(\L(m_+)) = \sum_{\cal T \in \Tab(\sd)} M(\cal T).$$
\label{thm3}
\end{thm}
\begin{proof}
Given Proposition \ref{p:bijpathtab}, this is immediate from Theorem \ref{mainthm1}.
\end{proof}

\subsection{Non-generic super skew diagrams} \label{non generic sec}
In this section we discuss non-generic skew diagram.

We call a column $j$ of diagram $\sd$ non-generic if
\begin{equation}
\# \{i\in \Z| (i,j) \in \sd, (i,j+1)\in \sd\}=2N.
\end{equation}

Let $\sd$ be a non-generic super skew diagram. Let column $j'$ be  such that it is  non-generic and 
all columns $j$ with $j>j'$ are generic. Define 
\begin{align*}
\lefteqn{(\lambda'/\mu') = \{(i,j)\in \sd | j\leq j'\}\cup }\\
&& \{(t_{j'}-r,j')| 1\leq r \leq l_{j'+1} - 2N +1\}\cup \{(i-1,j-1)| (i,j)\in \sd, j>j'+1\}.
\end{align*}
The following lemma is straightforward.
\begin{lemma}
The shape $(\lambda'/\mu')$ is a super skew diagram. The number of non-empty columns of $(\lambda'/\mu')$ is one less than that of $\sd$. The number of boxes of $(\lambda'/\mu')$ is $2N-1$ less than that of $\sd$. The number of non-generic columns of
$(\lambda'/\mu')$ is one less than that of $\sd$. \qed
\end{lemma}
We call the super skew diagram $(\lambda'/\mu')$ {\it closely related to} $\sd$. We call a generic super skew diagram $(\lambda'/\mu')$ {\it related} to a non-generic super skew diagram $\sd$ 
if there exists a sequence of super skew  diagrams $(\lambda_i/\mu_i)$, $1\leq i\leq S$, $S\in\Z_{\geq 1}$, such that 
 $(\lambda_1/\mu_1)=\sd$, $(\lambda_S/\mu_S)=(\lambda'/\mu')$ and $(\lambda_i/\mu_i)$ is closely related to $(\lambda_{i-1}/\mu_{i-1})$ for $2\leq i\leq S$. 
 
\begin{cor}
 Let $\sd$ be a non-generic super skew diagram. Then there exist unique generic skew diagram related to $\sd$. \qed
\end{cor}
Finally, we observe that the related super skew diagrams correspond to the same module over the affine quantum algebra.

\begin{prop}
Let super skew diagram $(\lambda'/\mu')$ be related to super skew diagram $\sd$. Then there exist a bijection $\tau: {\operatorname{Tab}}\sd \to {\operatorname {Tab}}(\lambda'/\mu')$ such that $M(\cal T)=M(\tau\cal T)$ for all $\cal T\in {\operatorname{Tab}}\sd$.
\end{prop}
\begin{proof}
It is sufficient to give such a bijection for closely related super skew diagrams.
Let $(\lambda'/\mu')$ be related to a non-generic super skew diagram $\sd$ and let $\cal T\in \operatorname{Tab}\sd$.
We define $\tau\cal T$ as follows:
$$(\tau\cal T)(i,j)=\begin{cases} 
\cal T(i,j), \hspace{70pt} &  j<j' \ {\rm or}\  j=j',\ i>b_j'-N, \\
\cal T(i+1,j+1), \hspace{40pt}& j>j' \ {\rm or}\ j=j',\ i<t_{j'}+N, \\
0,  \hspace{90pt} &  j=j'\ {\rm and}\  t_{j'}+N\leq i\leq b_{j'}-N,
\end{cases}$$
where $t_{j'}$ and $b_{j'}$ are, respectively, the top and the bottom box of the column $j'$ in $(\lambda'/\mu')$.
 
All the checks are straightforward.
\end{proof}

\end{document}